\documentclass[11pt, reqno, a4paper, final]{amsart}

\usepackage[english]{babel}              
\usepackage{amsmath}                     
\usepackage{amssymb}                     
\usepackage{amsthm}
\usepackage{color}
\usepackage{bbm}
\usepackage{hyperref}
\usepackage{marvosym}                   
\usepackage{a4wide}										 	
\usepackage{mathrsfs}									  
\usepackage[all]{xy}										
\usepackage{url}												
\usepackage{verbatim,epsfig}
\usepackage[notref,notcite,draft]{showkeys}   

\theoremstyle{plain}

\newtheorem{thm}{Theorem}[section] 
\newtheorem{theorem}{Theorem}[section] 
\newtheorem{cor}[thm]{Corollary}
\newtheorem{corollary}[thm]{Corollary}
\newtheorem{lem}[thm]{Lemma}
\newtheorem{lemma}[thm]{Lemma}
\newtheorem{prop-defi}[thm]{Definition \& Proposition}
\newtheorem{prop}[thm]{Proposition}
\newtheorem{por}[thm]{Porism}
\newtheorem*{thm*}{Theorem}
\newtheorem*{prop*}{Proposition}
\newtheorem*{cor*}{Corollary}
\newtheorem{proposition}[thm]{Proposition}

\theoremstyle{definition}
\newtheorem{defi}[thm]{Definition}
\newtheorem{definition}[thm]{Definition}

\newtheorem{rem}[thm]{Remark}
\newtheorem{remark}[thm]{Remark}

\newcommand{\NN}{{\mathbb N}}
\newcommand{\ZZ}{{\mathbb Z}}

\newcommand{\CC}{{\mathbb C}}

\newcommand{\FF}{{\mathbb F}}
\newcommand{\F}{{\mathcal F}}

\renewcommand{\P}{{\mathcal P}}

\newcommand{\G}{{\mathscr G}}

\newcommand{\B}{{\mathscr B}}

\renewcommand{\H}{\mathscr{H}}

\newcommand{\ip}[2]{\left\langle {#1}\hspace{0.05cm}, \hspace{0.05cm}{#2}\right \rangle}

\newcommand{\varps}{{\varepsilon}}
\newcommand{\rg}{{\operatorname{rg\hspace{0.04cm}}}}

\newcommand{\tens}{\otimes}

\newcommand{\spann}{{\operatorname{span}}}

\newcommand{\To}{\longrightarrow}

\newcommand{\supp}{{\operatorname{supp}}}

\newcommand{\Tor}{\operatorname{Tor}}

\newcommand{\id}{\operatorname{id}}

\newcommand{\Tr}{\operatorname{Tr}}

\renewcommand{\leq}{\leqslant}
\renewcommand{\geq}{\geqslant}
\newcommand{\twoone}{{\operatorname{II}_1}}
\newcommand{\mult}{{\operatorname{mult}}}
\newcommand{\bbb}{{\mathbbm{1}}}

\newcommand{\rtimescom}{{\bar{\rtimes}}}
\newcommand{\abexact}[2]{\phantom{ }_{#1\rightarrow #2}\textrm{exact}}
\newcommand{\abflat}[2]{\phantom{ }_{#1\rightarrow #2}\textrm{flat}}

\newcommand{\co}{{\operatorname{co}}}
\newcommand{\all}{{\operatorname{all}}}

\address{David Kyed,
Department of Mathematics,
KU Leuven,
Celestijnenlaan 200B,
B-3001 Leuven, 
Belgium.}
\email{David.Kyed@wis.kuleuven.be}
\urladdr{www.kuleuven.be/~u0078326}

\address{Henrik Densing Petersen, SB-MATHGEOM-EGG, EPFL, Station 8, CH-1015,
Lausanne, Switzerland.}
\email{henrik.petersen@epfl.ch}
\urladdr{www.math.ku.dk/~hdp}

\title{A groupoid approach to  L{\"u}ck's amenability conjecture}
\author{David Kyed}
\author{ Henrik D.~Petersen}

\keywords{Amenability, Groupoids,  Dimension flatness}
\subjclass[2010]{43A07, 18B40, 18G10, 46L10} 
\thanks{The research of the first named author is funded by The Danish Council for Independent Research $|$ Natural Sciences and the ERC Starting Grant VNALG-200749}
\thanks{The second named author acknowledges the support of the Danish National Research Foundation through the Centre for Symmetry and Deformation.}

\begin{document}

\begin{abstract}
We prove that amenability of a discrete group is equivalent to dimension flatness of certain ring inclusions naturally associated with measure preserving actions of the group. This provides a group-measure space theoretic solution to a conjecture of L{\"u}ck stating that amenability of a group is characterized by dimension flatness of the inclusion of its complex group algebra into the associated von Neumann algebra.

\end{abstract}

\maketitle

\section{Introduction}
The theory of $L^2$-invariants was re-formulated in terms of homological algebra by L{\"u}ck \cite{Luck97,Luck98} (see also \cite{farber}) in the late 1990's, and this prompted the importance of investigating the ring-theoretical properties of the group ring $\CC\Gamma$ associated with a discrete group $\Gamma$. In this context, one very natural question to ask is when the inclusion of  $\CC\Gamma$ into the group von Neumann algebra $L\Gamma$ is flat, but this turns out to be true only for a very limited class of groups: it is the case when $\Gamma$ is virtually cyclic and conjecturally \cite[Conjecture 6.49]{Luck02} only then. However, utilizing the von Neumann dimension $\dim_{L\Gamma}(-)$ arising from the natural trace on $L\Gamma$, one can relax the definition of flatness, arriving at the notion of dimension-flatness, which is simply defined by demanding that the functor $L\Gamma\otimes_{\CC\Gamma} -$ maps injective $\CC\Gamma$-homomorphisms to $L\Gamma$-homomorphisms with zero-dimensional kernel. This property turns out to be far less restrictive than actual flatness of the inclusion $\CC\Gamma\subseteq L\Gamma$, and in \cite[Theorem 5.1]{Luck98} L{\"u}ck proves that it holds for all amenable groups and conjectures this to be a characterization of amenability. In the present paper we investigate various group-measure space theoretic versions of this conjecture  --- henceforth referred to as \emph{L{\"u}ck's amenability conjecture}. Our primary setting will be that of translation groupoids arising from free probability measure preserving actions of discrete groups, and using  Gaboriau-Lyons' measure-theoretic solution to von Neumann's problem \cite{gaboriau-lyons} we prove the following:
\begin{thm*}[See Porism \ref{groupoid-porism}]
A discrete group $\Gamma$ is amenable if and only if the following holds: for any free, ergodic, probability measure-preserving action of $\Gamma$ on a non-atomic standard Borel space $(X,\mu)$ the inclusion of the corresponding groupoid ring $\CC[\G]$ into the  groupoid von Neumann algebra $L(\G)$ is dimension  flat.
\end{thm*}
Furthermore, we obtain the following version of the above result which is somewhat more group theoretical in nature:

\begin{thm*}[See Theorem \ref{amenability-to-dim-flat-thm}, \ref{non-flat-thm-all-borel-sets} \& Proposition \ref{fg-ground-ring}] 
A discrete countable group $\Gamma$ is amenable if and only if the following holds: For any finite cyclic group $C$ and any non-trivial system $\B$ of Borel subsets in $X:=\prod_{\Gamma}C$, which is stable under complements, finite intersections and the Bernoulli $\Gamma$-action, the  algebra $\CC[\B]$ generated by the corresponding  indicator functions satisfies that inclusion of the algebraic crossed product $\CC[\B]\rtimes \Gamma$ into the von Neumann algebraic crossed product $ L^\infty(X)\rtimescom \Gamma$ is dimension flat. Moreover, when $\Gamma$ is non-amenable the Borel system $\B$  for which the the inclusion is not dimension flat can be chosen such that $\CC[\B]$ has countable linear dimension and is finitely generated as a module over $\CC\Gamma$.
\end{thm*}

Our results are based on a detailed analysis of certain aspects of homological algebra ``relative to a dimension-function'' and along the way we prove several results of a general nature. When applied to the case of groupoid algebras they imply the following theorem, which in turn will be the key to the two dimension-flatness theorems mentioned above.

\begin{thm*}[See Corollary \ref{lma:flatbasechange} \& Proposition \ref{prop:dimflatness}]
If $\H$ is a sub-groupoid of a discrete measured groupoid $\G$ then the corresponding inclusion  $\CC\H\subseteq \CC\G$ is dimension flat relative to the von Neumann algebra of essentially bounded functions on their common base space. Furthermore, for any $\CC\H$-module $K$ and any $p\geq 0$ we have
\[
\dim_{L\G}\Tor_p^{\CC\H}(L\G, K)=\dim_{L\G}\Tor_p^{\CC\G}(L\G, \CC\G\underset{\CC\H}{\otimes}K)
\]
\end{thm*}
The rest of the paper is organized as follows. In Section \ref{sec:groupoids} we recall the necessary notions from the theory of discrete measured groupoids and introduce the module of functions on the homogenous space arising from an inclusion of such groupoids. This construction will turn out to be essential for the sections to come. In Section \ref{sec:homalg} we develop the homological algebraic results needed in order to obtain our main results which are proved in Section \ref{sec:amenability-to-dim-flat} and \ref{sec:back-again}.

\vspace{0.3cm}
\paragraph{\emph{Acknowledgements.}}
The authors would like to thank Ryszard Nest and Andreas Thom for valuable comments and  conversations revolving around L{\"u}ck's amenability conjecture. 

\vspace{0.3cm}
\paragraph{\emph{Notation.}}
Throughout the paper, all generic von Neumann algebras are assumed to be finite and have separable predual and, unless explicitly specified otherwise, $\tau$ will denote a fixed normal, faithful tracial state on the von Neumann algebra in question. Moreover, all generic discrete groups are implicitly assumed to be countable and all groupoids appearing will be assumed discrete and measured.
We denote the unit, either in a group or an algebra, by $\bbb$ and the indicator function on a set $F$ by $\bbb_F$.
For a function $f\colon X\rightarrow E$ into a vector space we denote $\supp f := \{ x\in X \mid f(x)\neq 0\}$, i.e. we do not automatically take the closure, even if $X$ might be a topological space. Finally, we will need to distinguish between algebraic and von Neumann algebraic crossed products;  the symbol ``$\rtimes$'' will therefore be used to denote the former while ``$\rtimescom$'' will denote the latter.

\section{Group actions and groupoids} \label{sec:groupoids}

Suppose that $\Gamma$ is a countable discrete group acting essentially freely and measure preservingly  on a standard diffuse (i.e.~without atoms) probability space $(X,\mu)$. Recall that the freeness assumption means that for every $\gamma \in \Gamma \setminus \{ e \}$ we have $\mu(\{x\in X \mid \gamma.x=x \} )=0$, and that preservation of the measure means that the push-forward measure $\gamma_{*}\mu$ equals $\mu$ for all $\gamma\in \Gamma$. The action of $\Gamma$ on $X$ defines a standard, measure-preserving equivalence relation $\mathcal{R} \subseteq X\times X$ by setting $x\sim_{\mathcal{R}} y$  if  there exists $\gamma \in \Gamma $ such that $y=\gamma.x$.
We may think of the relation $\mathcal{R}$ as a groupoid, called the translation groupoid of the action, and when doing so we often denote it by $\mathscr{G}$ instead of $\mathcal{R}$. The object space is $\mathscr{G}^{0} = X$, there is an arrow $(y,x)_{\mathscr{G}}$, where we often leave out the subscript ``$\mathscr G$'', from $x$ to $y$ exactly when they are related in $\mathcal{R}$, and the composition of arrows is given by $(z,y)\circ (y,x)=(z,x)$. On the relation we have the two natural projection maps $\textrm{pr}_i\colon X\times X \rightarrow X$ and these are exactly the target map $t=\textrm{pr}_1$ and source map $s=\textrm{pr}_2$ on $\mathscr{G}$. Recall \cite{FM1, sauer-betti-of-groupoids} that $s$ and $t$ give rise to a  groupoid measure $\nu$ on $\mathscr{G}$ by setting

\begin{equation}
\nu(A) = \int_X \#  (t^{-1}(y)\cap A) \; \mathrm{d}\mu(y) = \int_X \# (s^{-1}(x)\cap A) \; \mathrm{d}\mu(x); \nonumber
\end{equation}
the two integrals being equal because $\mathcal{R}$ preserves $\mu$.
Recall from \cite{sauer-betti-of-groupoids} that the \emph{groupoid ring} $\mathbb{C}\mathscr{G}$ of $\mathscr{G}$ is the subspace of  $L^{\infty}(\mathscr{G},\nu)$ consisting of (classes of) functions $f$ such that the functions $x\mapsto \# \{ \alpha \in t^{-1}(x) \mid f(\alpha) \neq 0 \}$ and $x\mapsto \# \{ \alpha \in s^{-1}(x) \mid f(\alpha) \neq 0\}$ are both essentially bounded on $X$. On $\CC\G$ we consider the \emph{convolution product} $(f*g)(\gamma) = \sum_{\alpha,\beta\in \mathscr{G}:\gamma=\alpha \beta} f(\alpha) g(\beta)$;  note that the sum is finite for almost every $\gamma\in \mathscr{G}$ so that the definition does in fact make sense. Furthermore, $\CC\G$ is equipped with an involution by setting $f^*(\gamma)=\overline{f(\gamma^{-1})}$ turning it into a unital $*$-algebra. Integration against the measure $\nu$ defines a faithful, positive trace $\tau$ on $\CC\G$ and the corresponding GNS-construction leads to an algebra of bounded operators whose weak closure (which is therefore a finite von Neumann algebra known as the \emph{groupoid von Neumann algebra}) will be denoted by $L\G$ in the following.\\

Our main aim in this section is to introduce a notion of ``the homogenous space'' associated with an inclusion of groupoids and to study its basic properties. Consider therefore a subgroupoid $\mathscr{H}$ of $\mathscr{G}$ with the same object space $\H^{0} = X$. For simplicity of notation we assume that $\mathscr{H}$, whence also $\mathscr{G}$, has infinite orbits on $X$ and furthermore that $[\mathscr{G}:\mathscr{H}]=\infty$ almost everywhere, i.e.~that there is no $\mathscr{G}$-invariant set $A\subseteq X$ of non-zero measure such that each $\mathscr{G}$-orbit on $A$ splits into finitely many $\mathscr{H}$-orbits on $A$. 
We also fix isomorphisms of measure spaces
\begin{equation}
\phi_{\mathscr{G}} \colon X\times \mathbb{N} \rightarrow \mathscr{G} \quad \mathrm{and} \quad \phi_{\mathscr{H}} \colon X\times \mathbb{N} \rightarrow \mathscr{H} \nonumber
\end{equation}
such that $t\circ \phi_{*} = \mathrm{pr}_X$ and (mainly for convenience) such that $\phi_*((x,1)) = (x,x)_{\mathscr{*}}=\id_x$. The existence of $\phi_*$  is implicit in \cite[Lemma 3.2]{SauerThom}, where a full proof is given in the ergodic case.  In the case where $\mathscr{G}$ is the translation groupoid of a free action we may in fact take the domain of $\phi_{\mathscr{G}}$ to be $X\times \Gamma$ and $\phi(x,\gamma) = (x,\gamma^{-1}.x)_{\mathscr{G}}$.  Now define
\begin{equation}
E_1 = \cup_{j\in \mathbb{N}} E_{1,j} \ \mathrm{ where} \quad E_{1,j} := \supp \left( \bbb_{\phi_{\mathscr{G}}(X\times \{1\})} * \bbb_{\phi_{\mathscr{H}}(X\times \{j\})} \right). \nonumber
\end{equation}
Assuming $E_1, \dots , E_{i-1}$ defined as $E_*:=\cup_{j\in \mathbb{N}} E_{*,j}$ we define, recursively, sets $E_{i,*}\subseteq \mathscr{G}$ as follows:
For each $x\in X$ denote by $n_i(x)$ the smallest $n_i(x)\in \mathbb{N}$ such that $(x,n_i(x))\notin E_1\cup \cdots \cup E_{i-1}$. Denote by $S_i$ the graph $\{ (x,n_i(x))\mid x\in X\} \subseteq X\times \mathbb{N}$ and put
\begin{equation}
E_i = \cup_{j\in \mathbb{N}} E_{i,j} \quad \mathrm{where} \quad E_{i,j} = \supp \left( \bbb_{\phi_{\mathscr{G}}(S_i)} *\bbb_{\phi_{\mathscr{H}}(X\times \{j\}}\right). \nonumber
\end{equation}
Then the sets $E_i^{x} := t^{-1}(x)\cap E_i$ are precisely the pointwise orbits of right-multiplication by $\mathscr{H}$ on $\mathscr{G}$, and the maps $x\mapsto n_i(x)$ provides a measurable choice of representatives. For notational convenience, we also denote the arrow $\phi_{\mathscr{G}}(x,n_i(x))$ by $\alpha_{i,x}$. Notice also that the sets $E_{i,j}$ are pairwise disjoint with $\nu(E_{i,j})=1$ and $\mathscr{G} = \cup_{i,j} E_{i,j}$. 
\begin{remark}
The construction of the sets $E_i$ might seem technical at first glance but the underlying idea is quite simple. The set $E_1$ simply consists of the arrows in $\H$. To construct $E_2$, we choose for every point $x\in X$ the first (relative to the chosen numbering $\phi_\G$) arrow in $ \G\setminus \H$ with target $x$. This is then $\alpha_{2,x}$ and $E_2$ then consists of all arrows that can be obtained by composing the $\alpha_{2,x}$'s from the right with arrows from $\H$. The set $E_3$ is then constructed by choosing, for each $x\in X$, the first arrow in $\G\setminus (E_1\cup E_2)$ with target $x$. This is the arrow denoted $\alpha_{3,x}$ and the set $E_3$ consists of all the arrows that can be obtained by  multiplying the $\alpha_{3,x}$'s from the right with arrows from $\H$. 
Note also that the set $E_{i,1}$ simply consists of the collection $(\alpha_{i,x})_{x\in X}$ and that $E_{ij}$ consists of the arrows obtained by composing $\alpha_{i,x}$'s from the right with the $j$'th arrow from $\H$ with target $s(\alpha_{i,x})$.
\end{remark}

We now define the quotient space as $\left( \mathscr{G}/\mathscr{H}\right)^{\phi_*}:= \cup_{i\in \mathbb{N}} E_{i,1}$ with the Borel structure inherited from $\mathscr{G}$. Notice that this is actually independent of the choice of $\phi_\H$ and $\phi_\G$, in the sense that any other choices would give a canonically isomorphic space  ---  hence we drop the superscript on $\G/\H$ in the sequel. Next we want to study certain modules of functions on $\G/\H$. To this end, consider the sets
\begin{align*}
\CC[\G]_t &=\{ f\in L^\infty(\G,\nu) \mid  x\mapsto \#t^{-1}(x)\cap \supp(f) \text{ is essentially bounded on $X$}   \};\\
\CC[\G/\H]_t &= \{ f\in L^\infty(\G/\H,\nu) \mid  x\mapsto \#t^{-1}(x)\cap \supp(f) \text{ is essentially bounded on $X$}   \}.
\end{align*}
Note that $\CC[\G]_t$ is a left $\CC[\G]$-module for the natural convolution action.
We may define a map $\kappa_{\mathscr{H}}^{\mathscr{G}} \colon \mathbb{C}\mathscr{G}_t \rightarrow \mathbb{C}\mathscr{G}_t$ by setting
\begin{equation}
\kappa_{\mathscr{H}}^{\mathscr{G}}(f)(\alpha) = \left\{ \begin{array}{cl} \sum_{\beta\in E_i^x} f(\beta) & ,\; \textrm{if} \;\exists x\in X ,i\in \mathbb{N} : \alpha=\alpha_{i,x}; \\ 0 &, \;\textrm{if not.}\end{array} \right. \nonumber
\end{equation}
We note that the range of $\kappa_\H^\G$ is exactly $\CC[\G/\H]_t$ and define $\mathbb{C}\left[ \mathscr{G}/\mathscr{H}\right] := \kappa_{\mathscr{H}}^{\mathscr{G}}\left( \mathbb{C}\mathscr{G}\right)$. 

\begin{prop-defi}\label{prop-defi-ting}
The space  $\mathbb{C}\left[ \mathscr{G}/\mathscr{H}\right]_t$ is endowed with the structure of a left $\mathbb{C}\mathscr{G}$-module by setting
\begin{equation}
f.\xi :=\kappa_{\mathscr{H}}^{\mathscr{G}}(f*\xi) \ \text{ for } f\in \mathbb{C}\mathscr{G} \text{ and } \xi \in \mathbb{C}\left[ \mathscr{G}/\mathscr{H}\right]_t. \label{eq:CGmodstructure}
\end{equation}
The subset $\CC[\G/\H]$ is a submodule for this structure and the map $\kappa_{\mathscr{H}}^{\mathscr{G}}$ is a $\mathbb{C}\mathscr{G}$-map of $\mathbb{C}\mathscr{G}_t$ onto $\mathbb{C}\left[ \mathscr{G}/\mathscr{H}\right]_t$ mapping $\CC[\G]$ onto $ \CC[\G/\H].$
\end{prop-defi}

\begin{proof}
To show that equation (\ref{eq:CGmodstructure}) does indeed define a module structure it is sufficient to show that $\kappa_{\mathscr{H}}^{\mathscr{G}}(f*g) = \kappa_{\mathscr{H}}^{\mathscr{G}} \left( f*\kappa_{\mathscr{H}}^{\mathscr{G}}(g)\right)$ for all $f\in \CC[\G]$ and $g\in \mathbb{C}\mathscr{G}_t$. Expanding this, we need to show that for all $i\in \mathbb{N}$ and almost every $x\in X$ we have
\begin{equation}
\sum_{\alpha \in E_i^{x}}(f*g)(\alpha) = \sum_{\alpha\in E_i^x} (f*\kappa_{\mathscr{H}}^{\mathscr{G}}(g))(\alpha). \label{kappa-eq}
\end{equation}
Computing the left-hand side of \eqref{kappa-eq} we get
\begin{equation}
\sum_{\alpha \in E_i^{x}}(f*g)(\alpha) = \sum_{\alpha \in E_i^x} \sum_{\beta\in t^{-1}(x)} f(\beta)g(\beta^{-1}\alpha) = \sum_{\beta\in t^{-1}(x)} f(\beta) \left( \sum_{\alpha\in E_i^x} g(\beta^{-1}\alpha)\right), \nonumber
\end{equation}
and the right hand side of \eqref{kappa-eq} expands as
\begin{equation}
\sum_{\alpha\in E_i^x} (f*\kappa_{\mathscr{H}}^{\mathscr{G}}(g))(\alpha) = \sum_{\beta\in t^{-1}(x)} f(\beta) \left( \sum_{\alpha\in E_i^{x}} (\kappa_{\mathscr{H}}^{\mathscr{G}}g)(\beta^{-1}\alpha) \right). \nonumber
\end{equation}
For fixed $\beta$ there exists a unique $j\in \NN$ such that $\alpha_{j,s(\beta)}\in \beta^{-1}E_i^x$ and since $\kappa_{\H}^{\G}(g)$ is only supported in the representatives we have 
\[
\sum_{\alpha\in E_i^{x}} \kappa_{\mathscr{H}}^{\mathscr{G}}(g)(\beta^{-1}\alpha)=
\kappa_{\H}^{\G}(g)(\alpha_{j,s(\beta)})
= \sum_{\gamma\in E_j^{s(\beta)}}g(\gamma)=\sum_{\alpha\in E_i^x}g(\beta^{-1}\alpha),
\]
The remaining statements follow directly from the definitions.

\end{proof}

\begin{remark} \label{rmk:augmentationmap}
Removing the assumption that $\mathscr{H}$ has infinite index in $\mathscr{G}$ we can write $X$ as a disjoint union of $\mathscr{G}$-invariant sets $X^{(n)}, n\in \mathbb{N}\cup \{ \infty \}$ such that for all $n$ we have $[\mathscr{G}\vert_{ X^{(n)}}:\mathscr{H}\vert_{ X^{(n)}}] = n$. Then we may proceed as above on each of the pieces $X^{(n)}$, getting sets $E_{i,j}^{(n)}$ and defining a factor map $\kappa_{\mathscr{H}}^{\mathscr{G}}$ and $\mathbb{C}\left[ \mathscr{G}/\mathscr{H}\right]$.
Note also that if we take $\mathscr{H}=\mathscr{G}$ we get $\mathbb{C}\left[ \mathscr{G}/\mathscr{H}\right] \simeq L^{\infty}(X)$ and that $\kappa_{\mathscr{G}}^{\mathscr{G}}$ is the usual augmentation map. Thus in this case the $\mathbb{C}\mathscr{G}$-module structure coincides with the one considered in \cite{sauer-betti-of-groupoids}.
\end{remark}
Below we will also need a further subdivision of the $E_{i,1}$. Namely, noting that the projection maps $s$ and $t$ are countable-to-one we can  partition each $E_{i,1}$  into sets on which both the target and the source map are injective (see e.g.~\cite[Lemma 3.1]{sauer-betti-of-groupoids}.) That is, we can find subsets $I_i\subseteq \NN$ such that $E_{i,1} = \sqcup_{l\in I_i} E_{(i,l),1}$, $\nu(E_{(i,l),1})>0$ and such that $s$ and $t$ are both injective when restricted to each $E_{(i,l),1}$. For $x\in t(E_{(i,l),1})$ we denote by $\alpha_{(i,l),x}$ the unique arrow $\alpha$ in $E_{(i,l),1}$ with $x=t(\alpha)$ and we denote by $\alpha_{(i,l)}$ the partial isomorphism of $X$ given by the collection of all these. We furthermore denote by $E_{(i,l),j}$ the support of $\bbb_{E_{(i,l),1}}*\bbb_{\phi_{\mathscr{H}}(X\times \{j\})}$ and observe that the source and target maps are still injective when restricted to the $E_{(i,l),j}$.\\

For an inclusion $H\leq G$ of groups, it is well-known that the group ring $\mathbb{C}G$ is generated as a right-$\mathbb{C}H$-module by a set of representatives for the cosets in $G/H$. We now fix the last bit of notation and prove the groupoid analogue of this result.
\begin{definition}
An element $f\in \mathbb{C}\mathscr{G}$ is said to be on $\mathscr{H}$-reduced form if there exists a finite set $J\subseteq \sqcup_{i\in \mathbb{N}} I_i$ and $(f_{i,l})_{(i,l)\in J}\subseteq  \mathbb{C}\mathscr{H}$ such that
\begin{equation} \label{eq:Hreduced}
f= \sum_{(i,l)\in J} \bbb_{E_{(i,l),1}} * f_{i,l} \quad \text{ and } \quad t(\supp(f_{i,l}))\subseteq s(E_{(i,l),1}) \text{ for all } (i,l)\in J. 
\end{equation}
\end{definition}
The condition on the $f_{i,l}$'s just means that we have not trivially extended their support. In particular, for an $f$ on $\mathscr{H}$-reduced form we have  that $f=0$ if and only if all the $f_{i,l}=0$. To see this, first note that summands in \eqref{eq:Hreduced} have disjoint support, so $f$ is zero if an only if $\bbb_{E_{(i,l),1}}\ast f_{i,l}=0$ for every $(i,l)\in J$. Furthermore, if $f_{i,l}$ is non-zero on a set $F\subseteq \H$ of positive measure then the targets of this set is contained in $s(E_{(i,l),1})$, and hence for each $\beta\in F$ there exists a unique $\alpha\in E_{(i,l),1}$ for which the product $\alpha\beta$ is defined; this product is then in the support of $\bbb_{E_{(i,l),1}}\ast f_{i,l}$ which therefore has positive measure as well.

\begin{lemma}[Decomposition] \label{lma:decomposition}
Let $f\in \mathbb{C}\mathscr{G}_t$ and $0<\varepsilon \leq 1$. Then there exists a set $Y\subseteq X$ with $\mu(Y)\geq 1-\varepsilon$ such that $\bbb_{Y}*f$ is on $\mathscr{H}$-reduced form.
\end{lemma}

An alternative formulation using the dimension function (see e.g.~Section \ref{sec:homalg}) is that the right $\mathbb{C}\mathscr{H}$-span of the indicator functions $\bbb_{E_{(i,l),1}}$ generates a rank dense (i.e.~codimension zero for $\dim_{L^{\infty}(X)}$) submodule of $\mathbb{C}\G_t$.

\begin{proof}
By the definition of $\mathbb{C}\mathscr{G}_t$ we have $\nu (\supp f)<
\infty$. If we put 
\[
F_{(i,l),j} = E_{(i,l),j} \cap \supp(f) \quad \text{ for } i,j\in
\mathbb{N} \text{ and } l\in I_i,
\]
we therefore have $\sum_{(i,l),j} \nu (F_{(i,l),j})< \infty$, so we
may choose a finite set $D\subseteq (\sqcup_{i\in \mathbb{N}}I_i)\times \mathbb{N}$ such that
$\sum_{((i,l),j)\notin D} \nu(F_{(i,l),j}) < \varepsilon$. As

\[
\mu(t(F_{(i,l),j}))= \int_{X} \bbb_{t(F_{(i,l),j})}(x)\; \mathrm{d}\mu(x)\leq \int_{X} \# (t^{-1}(x)\cap F_{(i,l),j} ) \; \mathrm{d}\mu(x)=\nu(F_{(i,l),j}),
\]
also $\sum_{((i,l),j)\notin D} \mu (t(F_{(i,l),j})) < \varepsilon$ and  we now choose $Y= \left(
\cup_{((i,l),j)\notin D} t(F_{(i,l),j}) \right)^{\complement}$.
Then $\supp (\bbb_Y * f) \subseteq \cup_{((i,l),j)\in D} E_{(i,l),j}$ and since $s$ is injective on $E_{(i,l),j}$ 
there exists $f_{(i,l),j}^0\in L^{\infty}(X), \; ((i,l),j)\in D$, such that
\begin{eqnarray}
\bbb_Y*f & = & \sum_{(i,l),j\in D} \bbb_{E_{(i,l),j}} \cdot f_{(i,l),j}^0 \nonumber \\
 &=& \sum_{(i,l),j\in D} \bbb_{E_{(i,l),1}} \ast \bbb_{\phi_{\H}(X\times\{j\})} \cdot f_{(i,l),j}^0 \nonumber \\
 & = & \sum_{(i,l)} \bbb_{E_{(i,l),1}} * \left( \bbb_{s(E_{(i,l),1})} \ast \sum_{j} 
\bbb_{\phi_{\mathscr{H}}(X\times \{j\})} \cdot f_{(i,l),j}^0 \right). \nonumber
\end{eqnarray}
The functions $f_{i,l}:= \bbb_{s(E_{(i,l),1})}\ast \sum_{j} 
\bbb_{\phi_{\mathscr{H}}(X\times \{j\})} \cdot f_{(i,l),j}^0 $ will therefore now do the job.
\end{proof}

We single out the following consequence of the proof, which does not use the existence of $\mathscr{H}$ at all.

\begin{por}
For every $f\in \mathbb{C}\mathscr{G}_t$ and every $0<\varepsilon <1$ there is a $Y\subseteq X$ with $\mu(Y) \geq 1-\varepsilon$ such that $\bbb_Y*f \in \mathbb{C}\mathscr{G}$.
\end{por}

\section{Homological algebra in the presence of a dimension  function} \label{sec:homalg}

In this section we study certain basic homological algebraic concepts, in particular flatness properties, replacing the usual notion of exactness with a weaker notion arising from the dimension  function associated with a finite von Neumann algebra. We remind the reader that all generic von Neumann algebras are assumed to be finite and have separable predual and furthermore to come equipped with a fixed faithful, normal, tracial state denoted by $\tau$.

\subsection{L{\"u}ck's dimension  function}

Let $N$ be a tracial von Neumann algebra with a fixed faithful, normal, tracial state $\tau$. Recall that L{\"u}ck's dimension function $\dim_{N}$ assigns to each $N$-module $L$ an extended positive real number 
\begin{equation}
\dim_{N}L := \sup\{ \dim_{N}P \mid P\subseteq L \; \textrm{finitely generated and projective submodule} \}, \nonumber
\end{equation}
where $\dim_N P$ is the usual von Neumann dimension of the projective module $P$. 
For more details we refer to the monograph \cite{Luck02};  recall, in particular, the many nice properties listed in \cite[Theorem 6.7]{Luck02}. A key technical observation that will be used repeatedly in the sequel, often referred to as \lq Sauer's local criterion\rq, provides a very nice characterization of zero-dimensional modules over $N$. We recall it here for the readers convenience.

\begin{theorem}[{\cite[Theorem 2.4]{sauer-betti-of-groupoids}}]\label{sauers-local-criterion}
Let $N$ be a tracial von Neumann algebra and let $L$ be an $N$-module. Then $\dim_{N}L=0$ if and only if for every $x\in L$ there exists a sequence of projections $p_n\in N$ such that $\lim_n\tau(p_n)=1$ and $p_n.x=0$ for all $n$.
\end{theorem}
Note that if $p_n\in N$ is a sequence such that $\tau(p_n)\to 1$ and $p_nx=0$, then there exists a sequence of projections $p_n'\in N$ increasing to $\bbb$ and such that $p_n'x=0$. In particular, zero-dimensionality is a property that is independent of the choice of trace state on $N$.

\subsection{Dimension flat basis change}
In this section we develop some of the basic properties of homological algebra ``relative to'' a finite von Neumann algebra. The main technical result obtained is a version of the well-known flat base change formula \cite[Proposition 3.2.9]{Weibel} in this setting.

\begin{definition}
Let $R$ be a unital ring containing a finite von Neumann algebra $M$.

\begin{enumerate}
\item[(i)] We say that a complex 
\begin{equation}
\dots \xrightarrow{d_{i+2}} P_{i+1} \xrightarrow{d_{i+1}} P_i \xrightarrow{d_{i\phantom{+0}}} P_{i-1} \xrightarrow{d_{i-1}} \dots \nonumber
\end{equation}
of $R$-modules is $\dim_{M}$-exact if its homology in each degree has $M$-dimension zero, i.e.~if $\dim_{M} (\ker d_{i-1} / d_i(P_{i+1}))=0$ for all $i$.\\
\item[(ii)] We say that $(P_i,d_i)_{i\in \mathbb{N}_0}$ is a projective $\dim_{M}$-resolution of an $R$-module $L$ if the $P_i$'s are all projective $R$-modules and there exists an $R$-homomorphism $P_0\to L$ such that the augmented the complex $P_*\rightarrow L \rightarrow 0$ is $\dim_{M}$-exact.\\
\end{enumerate}
\end{definition}

The following definition provides us with language to talk about these properties in very general situations.

\begin{definition}
Let $(M_1,\tau_{1})$ and $(M_2,\tau_{2})$ be tracial von Neumann algebras and let $F$ be a functor from the category of $M_1$-modules to the category of $M_2$-modules. We say that $F$ is $\abexact{(M_1,\tau_1)}{(M_2,\tau_2)}$ if the image under $F$ of any short $\dim_{(M_1,\tau_{1})}$-exact sequence of $M_1$ modules is $\dim_{(M_2,\tau_{2})}$-exact.
\end{definition}

Our main focus in the following will be on tensor functors and we therefore adapt the standard language from homological algebra to this setting.

\begin{definition}
Let $M\subseteq N$ be a trace preserving inclusion of finite von Neumann algebras and let $R$ be an intermediate $*$-algebra.
The inclusion $R\subseteq N$ is said to be $\abflat{M}{N}$ if $N\tens_R - $ is $\abexact{M}{N}$. If $M$ and $N$ are clear from the context we will often just refer to the inclusion $R\subseteq N$ as being dimension flat. 
\end{definition}
\begin{remark}
The notion of dimension flatness originates from L{\"u}ck's work in \cite{Luck98} where it is proven that the inclusion $\CC[\Gamma]\subseteq L(\Gamma)$ is $\abflat{\CC}{L(\Gamma)}$ whenever $\Gamma$ is an amenable group. Note also that dimension flatness of an inclusion $R\subseteq N$ is independent of the choice of faithful, normal, tracial state on $N$. This follows from Sauer's local criterion (Theorem \ref{sauers-local-criterion}) and the remarks following it.

\end{remark}

Next we recall from \cite{Thom06a, Thom06b} the notion of rank completion. Given a tracial von Neumann algebra $M$ and an $M$-module $L$ we define \emph{the rank} of an element $\xi \in L$ as 
\[
[\xi] := \inf \{ \tau(p) \mid p\in \text{Proj}(M), p\xi = \xi \}.
\]
This induces a uniform structure on $L$ and the Hausdorff completion, denoted $c_M(L)$, is again an $M$-module. 
Hence there is a canonical map $\mathfrak{c} \colon L\rightarrow c_M(L)$ and this turns out \cite[Theorem 2.7]{Thom06b} to be a $\dim_{M}$-isomorphism; i.e.~the sequence $0\rightarrow L\xrightarrow{\mathfrak{c}} c_M(L) \rightarrow 0$ is $\dim_{M}$-exact. A module $L$ is called rank complete if $\mathfrak{c}$ is an isomorphism; we remark that the dimension function $\dim_M(-)$ is faithful on the category of complete modules.

\begin{defi}\label{dim-compatibel-ring}
Let $M\subseteq N$ be a trace-preserving inclusion of finite von Neumann algebras. An intermediate $*$-algebra $R$ is said to be $M$-\emph{compatible} if for any $R$-module $L$ and any $r\in R$ the action of $r$ on $L$ is Lipschitz with respect to the rank metric arising from $M$.
\end{defi}
The compatibility property is a mild strengthening of the property considered in \cite[Lemma 1.2]{neshveyev-rustad}. The results in \cite{neshveyev-rustad} came to our attention after submission of the present paper, and the theory developed there provides a different approach to some of the results in this section. We provide a self-contained account for the convenience of the reader.

\begin{rem}\label{dim-compatible-example-rem}
Note that if $R$ is $M$-compatible and $f\colon K\to L$ is a homomorphism of $R$-modules then both $c_M(K)$ and $c_M(L)$ are naturally $R$-modules and $f$ is a contraction with respect to the rank-metric and therefore extends continuously to a map $c_M(f)\colon c_M(K)\to c_M(L)$ which is also an $R$-homo\-mor\-phism. We will primarily be interested in the following two situations: Firstly, if $\Gamma\curvearrowright (X,\mu)$ is a free p.m.p.~action then $R=L^\infty(X)\rtimes \Gamma$ is $L^\infty(X)$-compatible as a subalgebra of $L^\infty(X)\rtimescom \Gamma$ \cite[Lemma 4.4]{Thom06b}. Secondly, in the case of a groupoid $\G$ with object space $(X,\mu)$ the groupoid ring $\CC\G\subseteq L\G$ is $L^\infty(X)$-compatible.
This follows from from the fact that $\CC\G$ is spanned (algebraically) as an $L^\infty(X)$-module by partial isometries with range- and source projections in $L^\infty(X)$ \cite[Lemma 3.3]{sauer-betti-of-groupoids}; the details of the argument can be found in the proof of \cite[Lemma 4.8]{sauer-betti-of-groupoids}.

\end{rem}

We record a few more properties of the rank completion that will turn out useful in the sequel. These are slightly more technical versions of \cite[Lemmas 2.6 \& 2.8]{Thom06b}.\\

\begin{lemma} \label{completion-properties}
Let $M\subseteq N$ be a trace-preserving inclusion of finite von Neumann algebras and let $R\subseteq S$ be intermediate $M$-compatible $*$-algebras. Then the following holds. 
\begin{enumerate}
\item[(i)] The functor $c_M(-)$ maps $\dim_M$-exact complexes of $R$-modules to ($\dim_\CC$-) exact complexes of $R$-modules.
\item[(ii)] For any $R$-module $L$ the natural map $\operatorname{id}\otimes \mathfrak{c} \colon S\otimes_{R} L \rightarrow S\otimes_{R} c_M(L)$ is a $\dim_{M}$-isomorphism. In particular when $S=N$ this is a $\dim_N$-isomorphism as well.

\end{enumerate}
\end{lemma}

Since $M$ is always compatible over itself,  (ii) implies that if $S$ is an $M$-compatible $*$-algebra between $M$ and $N$ then $S$ is also compatible as an $M$-bimodule in the sense of \cite[Definition 4.6]{sauer-betti-of-groupoids}. That is, if $L$ is a zero-dimensional $M$-module then the same is true for $S\tens_M L$. Note that the latter property  is exactly the one studied in \cite[Lemma 1.2]{neshveyev-rustad}. 

\begin{proof} To prove (i), consider a $\dim_M$-exact complex $K\overset{f}{\to} L \overset{g}{\to} Q$ and the commutative diagram
\[
\xymatrix{
K\ar[r]^f \ar[d]_{\mathfrak{c}_K} & L \ar[r]^g \ar[d]_{\mathfrak{c}_L} &Q \ar[d]_{\mathfrak{c}_Q}\\
c_M(K) \ar[r]_{c_M(f)}& c_M(L) \ar[r]_{c_M(g)} & c_M(Q)
}
\]
We then need to prove that $\ker(c_M(g))\subseteq \rg(c_M(f))$. Since the category of complete modules is abelian \cite[Theorem 2.7]{Thom06b} it suffices to prove that $\ker(c_M(g))/\rg(c_M(f))$ has $M$-dimension zero. By Sauer's local criterion, we therefore have to prove that for every $x\in \ker(c_M(g))$ and every $\varps>0$ there exists a projection $p\in M$ such that $\tau(p^\perp)<\varps$ and $px\in \rg(c_M(f))$. First choose $p_1\in M$ with $\tau(p_1^\perp)<\varps/3$ such that $p_1x\in \rg(\mathfrak{c}_L)$ and choose $y\in L$ such that $\mathfrak{c}_L(y)=p_1x$. Then 
\[
\mathfrak{c}_Q g(y)=c_M(g)\mathfrak{c}_L(y)=c_M(g)(p_1x)=p_1c_M(g)(x)=0.
\]
As $\ker(\mathfrak{c}_Q)$ is zero-dimensional there exists $p_2\in M$ such that $\tau(p_2^\perp)<\varps/3$ and $0=p_2 g(y)=g(p_2y)$. By $\dim_M$-exactness of the upper row there exists $p_3\in M$ such that $\tau(p_3^\perp)<\varps/3$ and $p_3p_2y=f(z)$ for some $z\in K$. Putting $p=p_1\wedge p_2\wedge p_3$ we have $\tau(p^\perp)<\varps$ and
\[
px=pp_1x=p\mathfrak{c}_L(y)= p \mathfrak{c}_L(p_3 p_2 y)=p\mathfrak{c}_L(f(z))=c_M(f)(\mathfrak{c}_K(pz)),
\]
and the proof of (i) is complete. To prove (ii), consider the map $f\colon L\to S\otimes_R L$ given by $f(x)=1\tens x$. This is an $M$-linear map so it extends to a map $\bar{f}\colon c_M(L)\to  c_M(S\otimes_R L)$ which is $R$-linear. It therefore induces a map $\id\tens \bar{f}\colon S\tens_R c_M(L) \to S\tens_R c_M(S\tens_R L)$ which after composition with the multiplication map $S\tens_R c_M(S\tens_R L)\to c_M(S\tens_R L)$ yields an $S$-linear map $\tilde{f}\colon S\otimes_R c_M(L)\to c_M(S\tens_R L )$ making the following diagram commutative:
\[
\xymatrix{
S\tens_R c_M(L) \ar[r]^{\tilde{f}} & c_M(S\tens_R L)\\
S\tens_R L \ar[u]^{\id\tens\mathfrak{c}} \ar[ur]_{\mathfrak{c}} & 
}
\]
Since $\mathfrak{c}$ is $\dim_M$-injective the same is true for $\id\tens \mathfrak{c}$. To prove dimension-surjectivity we need the following observation: For every $s\in S$ and every sequence of projections $p_n\in M$ with $\tau(p_n)\to 1$ there exists a sequence of projections $q_n\in M$ with $\tau(q_n) \to 1$ such that $q_ns= q_nsp_n$. This follows easily from the compatibility assumption on $S$. To see this explicitly,
note that since multiplication with  $s$ is Lipschitz there exists a $C>0$ such that
\[
[sp_n^\perp]\leq C[p_n^\perp]=C\tau(p_n^\perp).
\]
Hence there exists $q_n\in M$ with  $\tau(q_n)>1-C\tau(p_n^\perp)-1/n$ and $q_n s p_n^\perp=0$. 
Hence $\tau(q_n)\to 1$ and $q_nsp_n=q_ns$.
Proving dimension-surjectivity of $\id\tens \mathfrak{c}$ is now straight forward: Given $\xi =\sum_{i=1}^m r_i\tens x_i\in S\tens_R\mathfrak{c}_M(L)$ and $\varps>0$ there exist projections $p_1^{(n)},\dots, p_m^{(n)}\in M$ such that $\lim_n\tau(p_i^{(n)})=1$ and  $p_i^{(n)}x_i=\mathfrak{c}(y_i^{(n)})$ for some $y_1^{(n)},\dots, y_m^{(n)}\in L $. By the observation, we can find projections $q_1^{(n)},\dots,q_m^{(n)}\in M$ such that $\lim_n\tau(q_i^{n})=1$ and $q_i^{(n)} r_ip_i^{(n)}=q_i^{(n)} r_i$. Then $q^{(n)}=\wedge_{i=1}^m q_i^{(n)}$ satisfies $\lim_n\tau(q^{(n)})=1$ and
\[
q^{(n)}\xi =\sum_{i=1}^m q^{(n)}r_i\tens x_i =\sum_{i=1}^m q^{(n)}r_i \tens p_i^{(n)}x_i= (\id\tens \mathfrak{c})\left( \sum_{i=1}^m q^{(n)}r_i \tens y_i\right)
\]
Hence $\id\tens\mathfrak{c}$ is dimension-surjective and the proof of (ii) complete.
\end{proof}

\begin{corollary} \label{cor:flatspec}
If $M\subseteq N$ is a trace preserving inclusion of finite von Neumann algebras and $R$ is an intermediate $M$-compatible $*$-algebra then the inclusion $R\subseteq N$ is $\abflat{M}{N}$ if and only if it is $\abflat{\CC}{N}$.
\end{corollary}
\begin{proof}
Clearly $\abflat{M}{N}$ness is stronger than $\abflat{\CC}{N}$ness. On the other hand, if $R\subseteq M$ is $\abflat{\CC}{N}$ and 
\[
0\To K\To L \To Q\To 0
\]
is a $\dim_M$-exact sequence, then by Lemma \ref{completion-properties} (i) the $M$-rank completion of this sequence is properly exact and hence
\[
0\To N\tens_Rc_M(K) \To N\tens_R c_M(L) \To N\tens_R c_M(Q) \To 0
\]
is $\dim_N$-exact. Applying Lemma \ref{completion-properties} (ii), this complex of $N$-modules is $\dim_N$-isomorphic to the complex
\[
0\To N\tens_R K \To N\tens_R L \To N\tens_R Q \To 0
\]
which is therefore also $\dim_N$-exact; i.e.~the inclusion $R\subseteq N$ is $\abflat{M}{N}$.

\end{proof}

The following lemma shows that  a projective $\dim_M$-resolution is as good as an honest projective resolution for computing $\Tor$ as long as we only care about the dimension (compare also with \cite[Lemma 1.4]{neshveyev-rustad}).

\begin{lemma} \label{lma:dimres}
Let $M\subseteq N$ be a trace-preserving inclusion of finite von Neumann algebras and let $R$ be an intermediate $M$-compatible $*$-algebra. Suppose that $P_*\rightarrow L \rightarrow 0$ is a projective $\dim_{M}$-resolution of the left $R$-module $L$. Then for all $i\geq 0$
\begin{equation}
\dim_{N} \Tor_i^{R} (N,L) = \dim_{N} H_i(N\otimes_R P_*). \nonumber
\end{equation}
\end{lemma}

\begin{proof}
Let $(P_i)$ be a projective $\dim_{M}$-resolution of $L$ as in the statement and let $(Q_i)$ be an honest projective resolution of $L$. 
By Lemma \ref{completion-properties}, both $c_M(P_i)$ and $c_M(Q_i)$ are therefore honest (not necessarily projective) $R$-resolutions of $c_M(L)$.
By the comparison theorem \cite[Theorem 2.2.6]{Weibel} (see also  \cite[Porism 2.2.7]{Weibel}) we therefore get chain maps $u_*,v_*$ making the following diagram, in which the arrows denoted $\mathfrak{c}$ are the canonical maps into rank completions, commute.
\begin{displaymath}
\xymatrix{ & P_* \ar[rr] \ar[dl]_{\exists u_*} \ar[dd]^<<<<<<<{\mathfrak{c}} & & L \ar[rr] \ar[dl]_{\id} \ar'[d][dd]^{\mathfrak{c}} & & 0 \\
 Q_* \ar'[r][rr] \ar[dr]_{\exists v_*} \ar[dd]^{\mathfrak{c}} & & L \ar[rr] \ar[dr]^{\mathfrak{c}} \ar[dd]^<<<<<<<<<<<<<<<{\mathfrak{c}} & & 0 \\
  & c_M(P_*) \ar'[r][rr] \ar[dl]_{c_M(u_*)} & & c_M(L) \ar[rr] \ar[dl]_{\id} & & 0 \\
 c_M(Q_*) \ar[rr] & & c_M(L) \ar[rr] & & 0
 }
\end{displaymath}
We first consider the upper part of the diagram. Since $\mathfrak{c}\colon P_*\to c_M(P_*)$ lifts $\mathfrak{c}\colon L\to c_M(L)$ and this lift is unique up to homotopy \cite[Porism 2.2.7]{Weibel} we obtain that the composition of induced maps
\begin{equation}
H_i(N\otimes_R P_*) \xrightarrow{\bar{u}_i} \underbrace{H_i(N\otimes_R Q_*)}_{=\Tor_i^R(N,L)} \xrightarrow{\bar{v}_i} H_i(N\otimes_R c_M(P_*)) \nonumber
\end{equation}
is the same map as $\bar{\mathfrak{c}} \colon H_i(N\otimes_R P_*) \rightarrow H_i(N\otimes_R c_M(P_*))$. The latter is a $\dim_N$-isomorphism since it is the map induced on homology by the chain map $\id\tens \mathfrak{c}\colon N\tens_R P_i \to N\tens_R c_M(P_i)$, which is a $\dim_N$-isomorphism by Lemma \ref{completion-properties}.
In particular $\dim_{N} \ker \bar{u}_i = 0$ for all $i$, proving the inequality ``$\geq$'' of the statement.
For the other inequality we note, similarly, that the composition of induced maps
\begin{equation}
H_i(N \otimes_R Q_*) \xrightarrow{\bar{v}_i} H_i(N \otimes_R c_M(P_*)) \xrightarrow{\overline{c_M(u_i)}} H_i(N\otimes_R c_M(Q_*)) \nonumber
\end{equation}
is the same map as $\bar{\mathfrak{c}}\colon H_i(N \otimes_R Q_*) \rightarrow H_i(N\otimes_R c_M(Q_*))$ so that $\dim_{N} \ker \bar{v}_i = 0$.
\end{proof}

\begin{corollary}[Dimension flat base change] \label{lma:flatbasechange}
Let $M\subseteq N$ be a trace-preserving inclusion of finite von Neumann algebras and let $R \subseteq S$ be intermediate $M$-compatible $*$-algebras. Suppose that the functor $S\otimes_{R} -$ from $R$-modules to $S$-modules is $\abexact{M}{M}$. Then for every $R$-module $L$ and every $i\in \NN_0$ we have
\begin{equation}
\dim_{N} \Tor_i^{R} (N,L) = \dim_{N} \Tor_i^{S}(N,S\otimes_{R}L). \nonumber
\end{equation}
\end{corollary}

\begin{proof}
Let $P_*\rightarrow L \rightarrow 0$ be a resolution of $L$ by free $R$-modules. Then
\begin{equation}
S\otimes_{R} P_* \rightarrow S\otimes_{R} L \rightarrow 0 \nonumber
\end{equation}
is a free $\dim_{M}$-resolution of $S\otimes_{R} L$ so the claim follows directly from the previous lemma.
\end{proof}

Returning to the case of just one intermediate $*$-algebra $R$, we have the following equivalent characterizations of $\abflat{M}{N}$ness of the ring inclusion $R\subseteq N$, which is nothing but a straight forward dimension-adapted version of a classical result in homological algebra; see e.g.~\cite[Exercise 3.2.1]{Weibel}.

\begin{prop}\label{dim-flatness-eq-prop}
For the tower $M\subseteq R\subseteq N$, where $R$ is $M$-compatible, the following are equivalent.
\begin{itemize}
\item[(i)] The inclusion $R\subseteq N$ is $\abflat{M}{N}$; i.e. the functor $N\otimes_R -$ is $\abexact{M}{N}$.
\item[(ii)]
For every $k\geq 1$ and every left $R$-module $K$ we have $\dim_N\Tor_k^{R}(N,K)=0$.
\end{itemize}
\end{prop}

Before returning to the case of groupoids we record a minor result which will turn out useful in the sections to come.

\begin{lemma}\label{alg-crossed-dim-flat}
Let $M \subseteq N$ be a trace-preserving inclusion of von Neumann algebras with an intermediate $*$-algebra $R$ such that the inclusion $R\subseteq N$ is $\abflat{M}{N}$, and let $\Gamma$ be a discrete countable group acting on $N$ and preserving $R$ globally. Then the inclusion $R\rtimes \Gamma\subseteq N\rtimes \Gamma$ is also $\abflat{M}{N}$.
\end{lemma}

\begin{proof}
Let $0 \rightarrow K \xrightarrow{\iota} L \xrightarrow{\pi} Q \rightarrow 0$ be a short $\dim_M$-exact sequence of $R\rtimes \Gamma$-modules. The statement in the lemma is then just the observation that we have a commutative diagram of $N$-modules
\begin{displaymath}
\xymatrix{ 0 \ar[r] & (N\rtimes \Gamma)\otimes_{R\rtimes \Gamma} K \ar[r]^{\operatorname{id}\otimes \iota} & (N\rtimes \Gamma)\otimes_{R\rtimes \Gamma} L \ar[r]^{\operatorname{id}\otimes \pi} & (N\rtimes \Gamma)\otimes_{R\rtimes \Gamma} Q \ar[r] & 0 \\
 0 \ar[u]^{\wr}_{\operatorname{id}} \ar[r] & N\otimes_R K \ar[u]^{\wr}_{\operatorname{id}} \ar[r]^{\operatorname{id}\otimes \iota} & N\otimes_R L \ar[u]^{\wr}_{\operatorname{id}} \ar[r]^{\operatorname{id}\otimes \pi} & N\otimes_R Q \ar[u]^{\wr}_{\operatorname{id}} \ar[r] & 0.\ar[u]^{\wr}_{\operatorname{id}} }
\end{displaymath}
\end{proof}

\begin{rem}\label{non-dense-inclusion}
Note that if an inclusion $R\subseteq N$ is dimension flat then for any tracial inclusion $N\subseteq \tilde{N}$ into another finite von Neumann algebra $\tilde{N}$ the inclusion $R\subseteq \tilde{N}$ is also dimension flat. This is due to the fact that the inclusion $N\subseteq \tilde{N}$ is faithfully flat and the functor $\tilde{N}\tens_N -$ is dimension preserving \cite[Theorem 6.29]{Luck02}. Hence, dimension flatness of an inclusion $R\subseteq N$ is equivalent to dimension flatness of the inclusion of $R$ into the von Neumann subalgebra it generates in $N$.
\end{rem}

\subsection{Applications to groupoids} \label{sec:applicationstogroupoids}
We  now return to the setup from Section \ref{sec:groupoids}. More precisely, we consider an inclusion of discrete measured groupoids $\mathscr{H} \leq \mathscr{G}$ defined on the same object space $(X,\mu)$ and we wish to apply the results from Section \ref{sec:homalg} to the following diagram of inclusions.

\begin{displaymath}
\xymatrix{  & \mathbb{C}\mathscr{G} \ar@{^{(}->}[r] & L\mathscr{G} \\ L^{\infty}(X) \ar@{^{(}->}[r] & \mathbb{C}\mathscr{H} \ar@{^{(}->}[u] \ar@{^{(}->}[r] & L\mathscr{H} \ar@{^{(}->}[u]  }
\end{displaymath}

The following result is the analogue of the well-known observation that for an inclusion of groups $H\leq G$ the functor $L \mapsto \mathbb{C}G \otimes_{\mathbb{C}H} L$ is exact from $\mathbb{C}H$-modules to $\mathbb{C}G$-modules, i.e. that the inclusion $\mathbb{C}H\subseteq \mathbb{C}G$ is flat.

\begin{proposition}[Dimension-flatness of $\mathbb{C}\mathscr{H}\leq \mathbb{C}\mathscr{G}$] \label{prop:dimflatness}
The tensor functor $\mathbb{C}\mathscr{G}\otimes_{\mathbb{C}\mathscr{H}} -$  is $\abexact{L^{\infty}(X)}{L^{\infty}(X)}$ from the category of $\CC\H$-modules to the category of $\CC\G$-modules.
\end{proposition}
For the proof the following observation will be convenient.
\begin{lemma}\label{slice-maps-lem}
The maps $\varphi_{i,l}\colon \CC\G\to \CC\H$ given by $\varphi_{i,l}(f)(\gamma):=(\bbb_{E_{i,l},1}^*\ast f)(\gamma)$ for $\gamma \in \H$ are right $\CC\H$-linear and satisfy
\[
\varphi_{i,l}(\bbb_{E_{(k,m),1}})= \begin{cases} \bbb_{s(E_{(i,l),1})} &\mbox{{if} } {(i,l)=(k,m)}\\
0 & \mbox{{otherwise}}. \end{cases}
\]
\end{lemma}
\begin{proof}
Let $f\in \CC\G$ and $g\in \CC\H$ be given. Consider $\gamma\in \H$ and put $x=t(\gamma)$. There is at most one arrow in $E_{(i,l),1}$ with source $x$. Assume first that this arrow $\alpha_{(i,l),y}\in E_{(i,l),1}$ exists.  Evaluating in $\gamma$ we now get 
\begin{align*}
\varphi_{i,l}(f\ast g)(\gamma) &=\sum_{\alpha\beta=\gamma} \bbb_{E_{(i,l),1}}(\alpha^{-1}) (f\ast g)(\beta)\\
&= (f\ast g)(\alpha_{(i,l),y}\gamma)\\
&=\sum_{\alpha\beta=\alpha_{(i,l),y}\gamma} f(\alpha)g(\beta)\\
&=\sum_{\beta\in \H}f(\alpha_{(i,l),y}\gamma\beta^{-1})g(\beta)
\end{align*}
On the other hand,
\begin{align*}
(\varphi_{(i,l)}(f)\ast g)(\gamma) &= \sum_{\alpha\beta=\gamma}(\bbb_{E_{(i,l),1}}^*\ast f )(\alpha) g(\beta)\\
&=\sum_{\stackrel{\alpha,\beta\in \H}{\alpha\beta=\gamma}} \sum_{\stackrel{\xi,\eta\in \G}{\xi\eta=\alpha}} \bbb_{E_{(i,l),1}}(\xi^{-1}) f(\eta)g(\beta)\\
&=\sum_{\stackrel{\alpha,\beta\in \H}{\alpha\beta=\gamma}}f(\alpha_{(i,l),y}\alpha)g(\beta)\\
&=\sum_{\beta\in \H} f(\alpha_{(i,l),y}\gamma\beta^{-1})g(\beta)
\end{align*}
In the remaining case, i.e.~when $E_{(i,l),1}\cap s^{-1}(x)$ is empty, both the above expressions are seen to be zero and we conclude that $\varphi_{i,l}$ is right $\CC\H$-linear.
The orthogonality relations follow in the same manner: If $\alpha_{(i,l),y}$ exists we have
\begin{align*}
\varphi_{i,l}(\bbb_{E_{(i,l),1}}^*\ast\bbb_{E_{(k,m),1}} )(\gamma)=\sum_{\alpha\beta=\gamma}\bbb_{E_{(i,l),1}}(\alpha^{-1})\bbb_{E_{(k,m),1}}(\beta)=\bbb_{E_{(k,m),1}}(\alpha_{(i,l),y}\gamma).
\end{align*}
The latter quantity is zero if $(k,m)\neq (i,l)$  and if $(k,m)=(i,l)$ it attains the value 1 exactly when $\gamma=\id_{s(\alpha_{(i,l),y})}=\id_x$, and is otherwise zero.
Again, the remaining case when $E_{(i,l),1}\cap s^{-1}(x)$ is empty runs similarly 
and the proof is complete.
\end{proof}

\begin{proof}[Proof of Proposition \ref{prop:dimflatness}.]
Suppose that the sequence of $\mathbb{C}\mathscr{H}$-modules
\begin{equation} \label{eq:dimflatness1}
0\To K \overset{\iota}{\To} L \overset{\pi}{\To} Q \To 0
\end{equation}
is $\dim_{L^{\infty}(X)}$-exact. By Corollary \ref{cor:flatspec} we may assume that the sequence \eqref{eq:dimflatness1} is in fact ($\dim_{\mathbb{C}}$-)exact. Then, since tensoring over a subring is always right-exact, it is enough to show that 
\[
\dim_{L^{\infty}(X)} \ker (\id\otimes \iota) = 0.
\]
For this we use Sauer's local criterion: Let $\xi =\sum_{r=1}^m f^{(r)}\tens x^{(r)} \in \ker(\id\tens \iota)$ and $0 < \varepsilon \leq 1$ be given. By the decomposition lemma we can find $Y_r\subseteq X$ such that $\mu(Y_r)>1-\varps/m$ and such that $\bbb_{Y_r}\ast f^{(r)}$ is on $\H$-reduced form. That is, there exists a finite set $D_r\subseteq (\sqcup_{i\in \NN} I_i)\times \NN$ and functions $f_{i,l}^{(r)}\in \CC\H$ such that
\[
\bbb_{Y_r}\ast f^{(r)}=\sum_{(i,l)\in D_r} \bbb_{E_{(i,l),1}}\ast f_{i,l}^{(r)}.
\]
By enlarging the expansion by zero-functions we may assume $D_1=\dots =D_m=:D$. Putting $Y:=\cap_{i=1}^m Y_r$ we have $\mu(Y)\geq 1-\varps$ and furthermore
\begin{align}\label{joint-reduced-form}
\bbb_Y.\xi &=\bbb_Y*\left (  \sum_{r=1}^m\left(\sum_{(i,l)\in D} \bbb_{E_{(i,l),1}}\ast f_{i,l}^{(r)} \right)\tens x^{(r)}  \right)\notag\\
&= \sum_{(i,l)\in D}(\bbb_Y\ast \bbb_{E_{(i,l),1} }) \tens \left ( \underbrace{\sum_{r=1}^m f_{i,l}^{(r)}x_r}_{=:y_{il}} \right ).  
\end{align}
As $\xi \in \ker(\id\tens\iota)$ we therefore have
\begin{align*}\label{reduced-form-eq}
0&=\bbb_Y.(\id\tens\iota)(\xi)\\
&=\bbb_Y.\left ( \sum_{(i,l)\in D}  \bbb_{E_{(i,l),1}}\tens  \iota(y_{i,l}) \right )\\
&=\sum_{(i,l)\in D} (\bbb_Y\ast \bbb_{E_{(i,l),1}})\tens \iota(y_{i,l})\\
&=\sum_{(i,l)\in D} (\bbb_{Y\cap t(E_{(i,l),1})}\ast \bbb_{E_{(i,l),1}})\tens \iota(y_{i,l})\\
&=\sum_{(i,l)\in D} \bbb_{E_{(i,l),1}}\ast \bbb_{\alpha_{i,l}^{-1}(Y\cap t(E_{(i,l),1}))}\tens \iota(y_{i,l})\\
&=\sum_{(i,l)\in D} \bbb_{E_{(i,l),1}}\tens \iota( \bbb_{\alpha_{i,l}^{-1}(Y\cap t(E_{(i,l),1}))}y_{i,l})
\end{align*}
Slicing the first leg with the maps from Lemma \ref{slice-maps-lem} we therefore obtain for every $(i,l)\in D$
\[
0=\iota(\bbb_{s(E_{(i,l),1})} \bbb_{\alpha_{i,l}^{-1}(Y\cap t(E_{(i,l),1}))}y_{i,l})=  \iota( \bbb_{\alpha_{i,l}^{-1}(Y\cap t(E_{(i,l),1}))}y_{i,l}),
\]
and thus $\bbb_{\alpha_{i,l}^{-1}(Y\cap t(E_{(i,l),1})}y_{i,l}=0$ for every $(i,l)\in D$. Performing the exact same manipulations in
\eqref{joint-reduced-form} we obtain
\[
\bbb_Y\ast \xi = \sum_{(i,l)\in D}(\bbb_Y\ast \bbb_{E_{(i,l),1} }) \tens \left( \sum_{r=1}^m f_{i,l}^{(r)}x_r \right)
=\sum_{(i,l)\in D} \bbb_{E_{(i,l),1}}\tens  \bbb_{\alpha_{i,l}^{-1}(Y\cap t(E_{(i,l),1}))}y_{i,l}=0. \qedhere
\]
\end{proof}

For an inclusion of groups $H\leq G$ one has $\mathbb{C}G \otimes_{\mathbb{C}H} \mathbb{C} \simeq \mathbb{C}[G/H]$. Again we have a similar result for groupoids which takes the following form.

\begin{proposition}\label{the-first-iso}
The composition
\begin{equation}
\mathbb{C}\mathscr{G} \otimes_{\mathbb{C}\mathscr{H}} L^{\infty}(X) \xrightarrow{{\mult}} \mathbb{C}\mathscr{G} \xrightarrow{\kappa_{\mathscr{H}}^{\mathscr{G}}} \mathbb{C}\left[ \mathscr{G}/\mathscr{H} \right] \nonumber
\end{equation}
is a $\dim_{L^{\infty}(X)}$-isomorphism. Here $\mult$ denotes the map $f\tens g\mapsto f\ast g $.
\end{proposition}
For the proof we need the following observation.

\begin{lem}\label{kappa-lem}
For all $i\in \NN$, $l\in I_i$ and $f\in \CC\H$ we have $\kappa_{\H}^{\G}(\bbb_{E_{(i,l),1}}\ast f )=\bbb_{E_{(i,l),1}}\ast \kappa_{\H}^{\H}(f)$, where both expressions are considered as functions on $\G$.
\end{lem}
The proof of Lemma \ref{kappa-lem} is a direct computation and we omit the details.

\begin{proof}[Proof of Proposition \ref{the-first-iso}]
Consider the augmentation map $\kappa_{\mathscr{H}}^{\mathscr{H}} \colon \mathbb{C} \mathscr{H} \rightarrow L^{\infty} (X)$ (See remark \ref{rmk:augmentationmap}). This fits into a short ($\dim_{\mathbb{C}}$-)exact sequence
\begin{equation}
0\To \underbrace{ \ker \kappa_{\mathscr{H}}^{\mathscr{H}}}_{=:K} \overset{\iota}{\To}  \mathbb{C}\mathscr{H} \overset{{\kappa_{\mathscr{H}}^{\mathscr{H}}}}{\To} L^{\infty}(X) \To 0. \nonumber
\end{equation}
Applying the functor $\mathbb{C}\mathscr{G}\otimes_{\mathbb{C}\mathscr{H}} -$ to this short exact sequence, we obtain the following commutative diagram in which the  upper sequence is $\dim_{L^\infty(X)}$-exact by Proposition \ref{prop:dimflatness} and the lower one is  ($\dim_{\mathbb{C}}$-)exact:
\begin{displaymath}
\xymatrix{ 0 \ar[r] & \mathbb{C}\mathscr{G}\otimes_{\mathbb{C}\mathscr{H}} K \ar[r]^{\id\tens\iota} \ar[d]^{\textrm{mult}} & \mathbb{C}\mathscr{G} \otimes_{\mathbb{C}\mathscr{H}} \mathbb{C}\mathscr{H} \ar[r]^{\id\otimes \kappa_{\mathscr{H}}^{\mathscr{H}}} \ar[d]^{\wr}_{\textrm{mult}} & \mathbb{C}\mathscr{G} \otimes_{\mathbb{C}\mathscr{H}} L^{\infty}(X) \ar[r] & 0 \\ 0 \ar[r] & \ker \kappa_{\mathscr{H}}^{\mathscr{G}} \ar[r] & \mathbb{C}\mathscr{G} \ar[r]^{\kappa_{\mathscr{H}}^{\mathscr{G}}} & \mathbb{C} \left[ \mathscr{G}/\mathscr{H} \right] \ar[r] & 0 }
\end{displaymath}
It is easy to see that the composition in the statement fits into this commutative diagram, so by the 5-lemma for dimension-isomorphisms \cite[page 3]{sauer-betti-of-groupoids}
it suffices to see that $\operatorname{mult} \colon \mathbb{C}\mathscr{G} \otimes_{\mathbb{C}\mathscr{H}} K \rightarrow \ker \kappa_{\mathscr{H}}^{\mathscr{G}}$ is a $\dim_{L^{\infty}(X)}$-isomorphism.
But this follows directly from the decomposition lemma   
and Sauer's local criterion: Take $f\in \ker \kappa_{\mathscr{H}}^{\mathscr{G}}$ and $\varepsilon > 0$ and choose $Y\subseteq X$ with $\mu(Y)\geq 1-\varepsilon$ such that $\bbb_Y*f$ is on $\mathscr{H}$-reduced form:
\[
\bbb_Y\ast f=\sum_{(i,l)\in D} \bbb_{E_{(i,l),1}}\ast f_{i,l} \quad \text{and} \quad t(\supp(f_{i,l}))\subseteq s(E_{(i,l),1}).
\]
Using Lemma \ref{kappa-lem} we now get
\begin{align*}
0=  \kappa_{\H}^{\G}(\bbb_{Y}* f)
= \kappa_{\H}^{\G}\left(  \sum_{(i,l)\in D} \bbb_{E_{(i,l),1}}\ast f_{i,l} \right)
=\sum_{(i,l)\in D} \bbb_{E_{(i,l),1}}\ast \kappa_{\H}^{\H}(f_{i,l}),   
\end{align*}
and as the $E_{(i,l),1}$'s are disjoint this implies $\bbb_{E_{(i,l),1}}\ast \kappa_{\H}^{\H}(f_{i,l})=0$ for all $(i,l)\in D$.  But since $t(\supp(f_{i,l}))\subseteq s(E_{(i,l),1})$ we get $\supp(\kappa_{\H}^{\H}(f_{i,l}))\subseteq s(E_{(i,l),1})$ and hence $\kappa_{\H}^{\H}(f_{i,l})=0$ for all $(i,l)\in D$; i.e.~$f_{i,l}\in K$. Thus
\[
\bbb_Y\ast f = \mult\left( \sum_{(i,l)\in D} \bbb_{E_{(i,l),1}}\otimes f_{i,l}\right)\in \mult(\CC\G\otimes_{\CC\H} K).
\]
This proves that $\mult$ is $\dim_{L^\infty(X)}$-surjective. That it is also $\dim_{L^\infty(X)}$-injective is clear since $\ker(\mult)$ is contained in the zero-dimensional module $\ker(\id\tens\iota)$.
\end{proof}

\section{From amenability to dimension flatness}\label{sec:amenability-to-dim-flat}
As mentioned in the introduction, we are interested in the dimension flatness of inclusions of the form $\CC[\Gamma] \subseteq L(\Gamma)$ for a discrete group $\Gamma$,  and, as was proven by L{\"u}ck \cite{Luck98}, this inclusion is dimension-flat if $\Gamma$ is amenable. 
More generally, we may ask for which subalgebras $R$ of $L(\Gamma)$ the inclusion $R\subseteq L(\Gamma)$ is dimension flat. In this section we provide partial  answers for subalgebras of $L(A_0\wr\Gamma)$ when $A_0$ is a finite cyclic group and $\Gamma$ is amenable.

\subsection{On wreath products with finite cyclic groups}\label{wreath-subsection}
Consider a finite cyclic group $A_0$ and put $A=\oplus_{\Gamma} A_0$. Recall that the \emph{wreath product } $A_0\wr \Gamma$ is defined as the semi-direct product $A\rtimes \Gamma$ where $\Gamma$ acts on $A$ by translations in the $\Gamma$-direction. 
Denote by $\hat{A}$ the  Pontryagin dual $\Pi_\Gamma \hat{A}_0$ of $A$,  and recall that the topology on $\hat{A}$ is generated by sets of the form
$\prod_{\gamma \in \Gamma} U_\gamma$,
where $U_\gamma \subseteq  \hat{A}_0$ and $U_\gamma= \hat{A}_0$ for all but finitely many $\gamma \in \Gamma$. By Tychonoff's theorem this turns $\hat{A}$ into a compact Hausdorff topological group, and by discreteness of $A_0$ the open sets in the canonical basis are all compact open. Note also that the compactness of $\hat{A}$ implies that every compact open set is a finite union of compact open sets from the basis. 
We denote by $\B_\co$ the family of compact open Borel subsets and by $\CC[\B_\co]$ the algebra generated by the corresponding indicator functions in $L^\infty(\hat{A})$. 
We briefly pause to remind the reader of the standard fact  that the algebra $\CC[\B_\co]$ exactly corresponds to the group algebra $\CC[A]$ under the Fourier transform:

\begin{lem}\label{basis-set-lem}
The Fourier transform $\F\colon L(A)\simeq L^\infty(\hat{A})$ maps $\CC[A]$ onto the subalgebra $\CC[\B_\co]$ generated by characteristic functions arising from compact open subsets in $\hat{A}$.

\end{lem}

For an amenable group $\Gamma$ it was proven by L{\"u}ck in  \cite{Luck98} that the inclusion $\CC\Gamma\subseteq L\Gamma$ is dimension flat. Since $A_0$ is a finite cyclic group also the wreath product  $A_0\wr \Gamma$ is amenable and the inclusion $\CC[A_0\wr \Gamma]\subseteq L(A_0\wr \Gamma)$ is therefore dimension flat as well. In the dual picture, this corresponds to dimension flatness of the inclusion $\CC[\B_\co] \rtimes \Gamma \subseteq L^\infty(\hat{A})\rtimescom \Gamma$. 
In the following we show that this is also the case for crossed products arising from other Boolean algebras than $\B_\co$. These results, however, are most naturally formulated in a general measure space theoretic setting, so we abandon the particular space $\hat{A}$ for the moment and consider instead an abstract standard Borel probability space. Before entering the discussion regarding dimension flatness let us fix a bit of notation.
\begin{defi}\label{boolian-defi}
Let $(X,\mu)$ be a standard, non-atomic Borel probability space. Denote by $\B_\all$ the system of all Borel subsets of $X$ and by $\B_\co$ the system of sets which are both open and closed. For any system $\B$ of Borel sets in $X$ which is stable under taking complements and finite intersections we denote by $\CC[\B]$ the linear span of the the indicator functions arising from $\B$.
\end{defi}
\begin{rem}
Note that when $\B$ is stable under complements and finite intersections then $\CC[\B]$ is a $*$-subalgebra of $L^\infty(X)$. Note also that both $\B_\all$ and $\B_\co$ have this property.
\end{rem}

\begin{thm}\label{commutative-dim-flatness}
Let $\B$ be a system of Borel sets in $X$ which is stable under complements and finite intersections and with the property that for any $\varps>0$ and any $A\in \B_\all$ there exists $B\in \B$ such that $\mu(A\triangle B)<\varps$. Then the inclusion $\CC[\B]\subseteq L^\infty(X)$  dimension flat.
\end{thm}
To prove Theorem \ref{commutative-dim-flatness} we will show that the inclusion $\CC[\B]\subseteq L^\infty(X)$ satisfies the strong F{\o}lner condition from \cite{amenability-and-vanishing}, and dimension flatness then follows from  \cite[Theorem 4.4]{amenability-and-vanishing}. For the convenience of the reader, we briefly recall the strong F{\o}lner condition before giving the proof of Theorem \ref{commutative-dim-flatness}. A weakly dense $*$-subalgebra $\mathcal{A}$ in a finite tracial von Neumann algebra $(M,\tau)$ is said to satisfy the strong F{\o}lner condition (see \cite[Proposition 3.3.]{amenability-and-vanishing}) if the following holds: For any $T_1,\dots, T_r \in \mathcal{A}$ there exists a sequence $\mathcal{S}_n\subseteq \mathcal{P}_n$ of non-zero finite dimensional subspaces in $\mathcal{A}$  such that the following holds
\begin{itemize}
\item[(i)] For every $i\in \{1,\dots, r\}$ and every $n\in \NN$ we have $T_i(\mathcal{S}_n)\subseteq \mathcal{P}_n$.
\item[(ii)] $\lim_{n\to \infty} \frac{\dim_\CC(\mathcal{S}_n)}{\dim_{\CC}(\mathcal{P}_n)}=1$.
\item[(iii)] The sequence of states $\varphi_{\mathcal{P}_n}\colon M\to \CC$ given by $\varphi_{\mathcal{P}_n}(T)=\frac{\Tr(P_nTP_n )}{\dim_\CC(\mathcal{P}_n)}$ converges in norm to the trace $\tau$. Here $P_n$ denotes the projection onto the subspace $\mathcal{P}_n$ and $\Tr$ denotes the semifinite trace in $B(L^2(M,\tau))$.
\end{itemize}

\begin{proof}[Proof of Theorem \ref{commutative-dim-flatness}]
First note that the assumption that every Borel set can be approximated arbitrarily well in measure by a set from $\B$ implies that $\CC[\B]$ is strongly dense in $L^\infty(X)$. To see this, it is enough to show that every projection $\bbb_{F}\in L^\infty(X)$ is in the strong operator closure of $\CC[\B]$. But the assumption on $\B$ implies that we can find a sequence of projections $\bbb_{F_{n}}\in \CC[\B]$ converging in 2-norm to $\bbb_F$, and since the strong operator topology coincides with the 2-norm topology on the unit ball of $L^\infty(X)$, it follows that $\CC[\B]$ is strongly dense in $L^\infty(X)$. Thus, we are in the setup from \cite{amenability-and-vanishing} and we now prove that the inclusion $\CC[\B]\subseteq L^\infty(X)$ satisfies the strong F{\o}lner condition

Let $T_1,\dots, T_r\in \CC[\B]$ be given  and assume, without loss of generality, that $\|T_i\|_\infty\leqslant 1$. Choose a sequence $\delta_n\in ]0,1]$ converging to zero.  
Since $\B$ is stable under finite intersections we can find a partition $F_1,\dots, F_s \in \B $ of $X$ such that each $F_i$ has positive measure and such that $T_1,\dots, T_r \in \spann_\CC\{\bbb_{F_i}\mid 1\leqslant i\leqslant s\}$.  Now choose, for each $i\in\{1,\dots, s\}$, a positive rational number $p_n^{(i)}/q_n$ such that $0\leq \mu(F_i)-p_n^{(i)}/q_n:=\delta_n^{(i)} <{\delta_n}/{2s}$ (we choose a common denominator right away)
as well as  a Borel set $H_1^{(i,n)}\subseteq F_i$ of measure $1/q_n$. 
By the assumptions made on $\B$, we can find $G_1^{(i,n)}\in \B$ such that $\mu(H_1^{(i,n)}\triangle G_1^{(i,n)})<\delta_n^{(i)}/p_nq_n$, where $p_n$ denotes the sum $\sum_{i=1}^{s} p_{n}^{(i)}$.  Upon replacing $G_1^{(i,n)}$ with $G_1^{(i,n)}\cap F_i$ we may furthermore assume that $G_1^{(i,n)}\subseteq F_i$. Moreover, since $\mu(H_{1}^{(i,n)})=1/q_n$ we have
\[
\left|\mu(G_{1}^{(i,n)})-\frac{1}{q_n} \right |<\frac{\delta_n^{(i)}}{p_nq_n}
\]
and hence
\[
\mu(F_i\setminus G_1^{(i,n)})\geq \mu(F_i)-1/q_n -\delta_n^{(i)}/p_nq_n =\delta_n^{(i)}-\delta_n^{(i)}/p_nq_n +(p_n^{(i)}-1)/q_n \geq (p_n^{(i)}-1)/q_n.
\]
So, if $p_n^{(i)}>1$ we can repeat the construction with $F_i$ replaced by $F_i\setminus G_1^{(i)}$ and iterating this process we obtain $p_n^{(i)}$ disjoint subset $G_1^{(i,n)},\dots ,G_{p_{n}^{(i)}}^{(i,n)}\in \B $ of $F_i$ such that
\[
\mu(F_i\setminus \cup_{j=1}^{p_n^{(i)}} G_{j}^{(i,n)} )\leq \mu(F_i)-p_n^{(i)}(1/q_n + \delta_n^{(i)}/p_nq_n)\leq 2\delta_n^{(i)}\leq \delta_n/s.
\]
Relabeling the $G_j^{(i,n)}$'s as $G_1^{(n)},\dots, G_{p_n}^{(n)}$ (same $p_n$ as above) and denoting $X\setminus \cup_{j=1}^{p_n} G_j^{(n)}$ by $G_0^{(n)}$ we have now obtained a family $G_0^{(n)},\dots, G_{p_n}^{(n)}\in \B$ such that:
\begin{itemize}
\item[(i)]  $G_i^{(n)}\cap G_j^{(n)}=\emptyset $ when $i\neq j$;
\item[(ii)] For $1\leqslant i\leqslant p$ we have $|\mu(G_i^{(n)})-\frac{1}{q_n}|\leq \frac{\delta_n}{2sp_nq_n}$ 
\item[(iii)]$ \mu(G_0^{(n)}) =\sum_{i=1}^s \mu(F_i\setminus \cup_{j=1}^{p_n^{(i)}} G_{j}^{(i,n)} )  \leq\delta_n.$
\end{itemize}
Now define 
\[
\P_n=\spann_\CC \{\bbb_{G_j^{(n)}}\mid 1\leq j\leq p_n\} \subseteq \CC[\B].
\]
Since the $G_j^{(n)}$'s are disjoint this is a $p_n$-dimensional subspace and since each $G_j$ is contained in exactly one $F_i$ the operators $T_i$ map $\P_n$ into itself. To see that $\CC[\B]\subseteq L^\infty(X)$ satisfies the strong F{\o}lner condition we therefore need to see that the sequence $\varphi_{\P_n}$ converges to $\tau$ in norm. \\
Since the characteristic functions $\bbb_{G_1^{(n)}},\dots, \bbb_{G_p^{(n)}}$ are orthogonal and span $\P_n$, by normalizing them we obtain an orthonormal basis and for  $T\in (L^\infty(X))_1$ we therefore have
\begin{align*}
\left|  \tau(T)-\varphi_{\P_n}(T)     \right| &=
\left|  \int_X T\; \mathrm{d}\mu-\frac{1}{p_n}\sum_{j=1}^{p_n} \mu(G_j^{(n)})^{-1}\ip{T\bbb_{G_j^{(n)}}}{\bbb_{G_j^{(n)}}}     \right| \\
&\leq \left| \int_XT\; \mathrm{d}\mu  -\frac{1}{p_n}\sum_{j=1}^{p_n}q_n\int_X T\bbb_{G_j^{(n)}}\; \mathrm{d}\mu  \right| +\\
&+\left| \frac{1}{p_n}\sum_{j=1}^{p_n}q_n\int_X T\bbb_{G_j^{(n)}}\; \mathrm{d}\mu -  \frac{1}{p_n}\sum_{j=1}^{p_n} \mu(G_j^{(n)})^{-1}\ip{T\bbb_{G_j^{(n)}}}{\bbb_{G_j^{(n)}}}          \right| \\
&\leq \left|\int_X T\; \mathrm{d}\mu -\frac{q_n}{p_n} \int_{X\setminus G_0^{(n)}} T\; \mathrm{d}\mu \right| +\\
&+\frac{1}{p_n}\sum_{j=1}^{p_n}\left|q_n-\mu(G_j^{(n)})^{-1})\right|\left|\ip{T\bbb_{G_j^{(n)}}}{\bbb_{G_j^{(n)}}}\right|\\
&\leq \mu(G_0^{(n)}) + \left|1-\frac{q_n}{p_n}\right| +\frac{1}{p_n}\sum_{j=1}^{p_n}\left| q_n-\mu(G_j^{(n)})^{-1}\right| \left\|\bbb_{G_{j}^{(n)}}\right\|^2_2\\
&\leq \delta_n + \left|1-\frac{q_n}{p_n}\right| +\frac{q_n}{p_n}\sum_{j=1}^{p_n}\left| \mu(G_j^{(n)})-\frac{1}{q_n} \right|\\
&\leq \delta_n + \left|1-\frac{q_n}{p_n}\right| +\frac{q_n}{p_n} p_n \frac{\delta_n}{2sp_nq_n}.
\end{align*}
The latter expression is independent of $T$ and goes to zero since $\delta_n\to 0$ and $1-p_n/q_n \leqslant\delta_n/2$.
\end{proof}

\begin{rem}
Note that when $X=\hat{A}$ (the dual of the infinite torsion group $A=\oplus_\Gamma A_0$ from before) the conditions in Theorem \ref{commutative-dim-flatness} are fulfilled for every  $\B$ containing $\B_\co$ (and being stable under complements and finite intersections). This follows from the regularity of $\mu$.\\

\end{rem}

We also record the following version of the dimension  flat base change formula.

\begin{cor}\label{dim-flat-base-change-v2}
Let $\Gamma \curvearrowright (X,\mu)$ be a probability measure preserving action and suppose that $\mathscr{B}\subseteq \mathscr{B}_{\all}$ is a $\Gamma$-stable system of Borel sets satisfying the assumptions in Theorem \ref{commutative-dim-flatness}. Then
\begin{align*}
&\dim_{L^\infty(X)\rtimescom \Gamma}\Tor_p^{\CC[\B]\rtimes \Gamma}\left(L^\infty(X)\rtimescom \Gamma, K\right)= \\
&\dim_{L^\infty(X)\rtimescom \Gamma}\Tor_p^{L^\infty(X)\rtimes \Gamma}\left(L^\infty(X)\rtimescom \Gamma, (L^\infty(X)\rtimes \Gamma)\underset{\CC[\B]\rtimes \Gamma}{\tens}K\right)
\end{align*}
for every $\CC[\B]\rtimes\Gamma$-module $K$ and every $p\geq 0$.
\end{cor}
Note that the conditions in Corollary \ref{dim-flat-base-change-v2} are satisfied as soon as the system $\mathscr B$ is stable under complements, finite intersections and the $\Gamma$-action and furthermore contains a set $B$ which is neither null or co-null and for which $\bbb_B$ has full central support in $L^\infty(X)\rtimescom \Gamma$.

\begin{proof}
The assumption on $B$ implies that $\B$ satisfies the assumptions in Theorem \ref{commutative-dim-flatness} and by Lemma \ref{alg-crossed-dim-flat} the inclusion $\CC[\B]\rtimes \Gamma \subseteq L^\infty(X)\rtimes \Gamma$ is therefore $\dim_{L^\infty(X)}$-flat. 
The statement now follows from Lemma \ref{lma:dimres} in the following way: Choose a free $\CC[\B]\rtimes \Gamma$-resolution $F_*\overset{d_*}{\to} K\to 0 $. Then the induced complex
\[
L^\infty(X)\rtimes \Gamma\underset{\CC[\B]\rtimes \Gamma}{\tens} F_* \xrightarrow{\id \tens d_*} L^\infty(X)\rtimes \Gamma\underset{\CC[\B]\rtimes \Gamma}{\tens} K\To 0
\]
is  $\dim_{L^\infty(X)}$-exact, and by Lemma \ref{lma:dimres} we have
\begin{align*}
&\dim_{L^\infty(X)\rtimescom \Gamma}\Tor_p^{L^\infty(X)\rtimes \Gamma}\left(L^\infty(X)\rtimescom \Gamma, (L^\infty(X)\rtimes \Gamma)
\underset{\CC[\B]\rtimes \Gamma}{\tens} K \right) =\\
 &\dim_{L^\infty(X)\rtimescom \Gamma}H_p\left( L^\infty(X)\rtimescom \Gamma\underset{L^\infty(X)\rtimes \Gamma}{\tens} (L^\infty(X)\rtimes \Gamma)\underset{\CC[B]\rtimes \Gamma}{\tens} \ , \  \id\tens\id\tens d_*\right) = \\
&\dim_{L^\infty(X)\rtimescom \Gamma}\Tor_p^{\CC[\B]\rtimes\Gamma}\left(L^\infty(X)\rtimescom \Gamma \ , K\right).
\end{align*}

\end{proof}

\begin{thm} \label{amenability-to-dim-flat-thm}
Let $(X,\mu)$ be a standard probability space without atoms and let $\Gamma$ be an amenable group acting freely
and measure preservingly on $X$.
If $\B$ is a family of Borel subsets satisfying the assumptions in Theorem \ref{commutative-dim-flatness} then
 the inclusion $\CC[\B]\rtimes \Gamma \subseteq L^\infty(X)\rtimescom \Gamma$ is dimension flat.
\end{thm}
\begin{proof}
Since $\Gamma$ is amenable, by \cite[Corollary 6.6]{amenability-and-vanishing} the inclusion $L^\infty(X)\rtimes \Gamma\subseteq  L^\infty(X)\rtimescom \Gamma$ is dimension flat and applying Corollary \ref{dim-flat-base-change-v2} we obtain
\begin{align*}
&\dim_{L^\infty(X)\rtimescom \Gamma} \Tor_p^{\CC[\B] \rtimes \Gamma}(L^\infty(X)\rtimescom \Gamma, K) = \\
& \dim_{L^\infty(X)\rtimescom \Gamma} \Tor_p^{L^\infty(X) \rtimes \Gamma}(L^\infty(X)\rtimescom \Gamma, L^\infty(X)\rtimes \Gamma \underset{\CC[\B]\rtimes\Gamma}{\otimes} K)=0
\end{align*}
for any $\CC[\B]\rtimes\Gamma$-module $K$ and any $p\geq 1$.
\end{proof}

\section{From dimension  flatness to amenability}\label{sec:back-again}
The aim of this section is to prove a converse to the statement in Theorem \ref{amenability-to-dim-flat-thm}. That is, we aim to show that if $\Gamma$ is a non-amenable group then there exists a non-atomic Borel probability space $(X,\mu)$ and a free p.m.p.~action $\Gamma \curvearrowright (X,\mu)$ and a stable family $\B$ of Borel sets in $X$  such that the inclusion $\CC[\B]\rtimes \Gamma\subseteq L^\infty(X)\rtimescom \Gamma$ is not dimension flat. 
More precisely, we will show that there exists
a finite cyclic group  $A_0$ and a stable system of Borel subsets in $\hat{A}$ (as before $A$ denotes $\oplus_\Gamma A_0$ and $\hat{A}$ its Pontryagin dual) such that the inclusion $\CC[\B] \rtimes \Gamma \subseteq L^\infty(\hat{A})\rtimescom \Gamma$ is not dimension  flat (see e.g.~Definition \ref{boolian-defi} for details on the terminology and notation). The crucial ingredient is Gaboriau-Lyons' striking  ``measure theoretic converse'' to von Neumann's problem which we recapitulate in the following.
\subsection{Gaboriau-Lyons' theorem}
The main result in \cite{gaboriau-lyons} is the following
\begin{thm}[\cite{gaboriau-lyons}]\label{gaboriau-lyons-continuous}
If $\Gamma$ is a countable discrete non-amenable group then the orbit equivalence relation of the Bernoulli action $\Gamma \curvearrowright [0,1]^\Gamma$ contains the orbit equivalence relation of an essentially free action of $\FF_2$.
\end{thm}

For our purposes the following discrete-base-space version will also turn out  relevant.
\begin{theorem}[\cite{gaboriau-lyons}]\label{gaboriau-lyons-discrete}
Let $\Gamma$ be a finitely generated non-amenable group. Then there is an $n\in \mathbb{N}$ and a non-empty open interval $(p_1,p_2)\subseteq [0,1]$ such that for every $p\in (p_1,p_2)$ there is an essentially free, ergodic action $\sigma$ of $\mathbb{F}_2$ on $\prod_1^n \left( \{ 0,1\}, \mu_p\right)^{\Gamma}$ such that the orbit equivalence relation $\mathcal{R}_{\sigma}$ is contained (almost everywhere) in the orbit equivalence relation $\mathcal{R}_{\Gamma}$ of the diagonal Bernoulli action.
\end{theorem}

We elaborate on the proof of \cite[Corollary  4]{gaboriau-lyons} in order to get the $\Gamma$-action in a more convenient form. First note that we have a $\Gamma$-equivariant isomorphism of measure spaces
\begin{eqnarray}
\varphi \colon \left( \{ 0,1\}^n, \mu_p^{\otimes n} \right)^{\Gamma} & \rightarrow & \prod_1^n \left( \{0,1\} , \mu_p \right)^{\Gamma} \nonumber \\
 &  & \varphi(x)(k)(\gamma) = x(\gamma)(k), \quad 1\leq k\leq n, \gamma\in \Gamma. \nonumber
\end{eqnarray}

Thus we may assume that we have our $\mathcal{R}_{\sigma}$ on $Y_0^{\Gamma}$ where $Y_0=\{0,1\}^n$ with the product measure $\mu_p^{\otimes n}$. Next we may assume that $p=\frac{s}{t}$ is a rational number, $s,t\in \mathbb{N}$. Then there is a surjective (but not necessarily injective) measure-preserving map $\psi_0 \colon \mathbb{Z}/t^n\mathbb{Z} \rightarrow Y_0$ where the domain is equipped with the equi-distributed probability measure $\nu_{t^n}$. This, in turn, then induces a measure-preserving, $\Gamma$-equivariant map $\psi \colon \left( X_0:= \mathbb{Z}/t^n\mathbb{Z} , \nu_{t^n}\right)^{\Gamma} \rightarrow (Y_0, \mu_p^{\otimes n})^{\Gamma}$.
We now use $\psi$ to pull back the action of $\mathbb{F}_2$ as follows. Write $X=X_0^{\Gamma}$, $Y=Y_0^{\Gamma}$, and $\mathbb{F}_2=\langle a,b\rangle$. Then there are measurable partitions $Y=\bigsqcup_{\gamma \in \Gamma} A_{\gamma} = \bigsqcup_{\gamma \in \Gamma} B_{\gamma}$ such that for all $\gamma \in \Gamma$ and (almost) all $y\in A_{\gamma}$ we have $a.y=\gamma.y$, and similarly for all $y\in B_{\gamma}$ we have $b.y=\gamma.y$. Now define partitions of $X$ by taking pre-images $A^o_{\gamma} := \psi^{-1}(A_{\gamma})$ and $B^o_{\gamma}:=\psi^{-1}(B_{\gamma})$. We get an action $\sigma^o$ of $\mathbb{F}_2$ on $X$ by defining $\sigma^o(a).x = \gamma.x$ for $ x\in A^o_{\gamma}$ and similarly for $b$. This clearly gives two well-defined measure-isomorphisms
since $\Gamma$ acts by measure-isomorphisms, whence an action since $\mathbb{F}_2$ is a free group. Finally, the action is essentially free since if $\sigma^o(w)\vert_{Z^o} = \id\vert_{Z^o}$ for some set $Z^o\subseteq X$ and some $w\in \mathbb{F}_2$ we would have $\sigma(w)\vert_{\psi(Z^o)}=\id\vert_{\psi(Z^o)}$ and hence $\psi(Z^o)$ has measure zero in $Y$. But $Z^o\subseteq \psi^{-1}(\psi(Z^o))$ and since $\psi$ is measure preserving, $Z^o$ must have measure zero in $X$. Note also that $\psi$ is an $\FF_2$-equivariant map by construction of the $\FF_2$-action on $X$. \\

We summarize all this as:

\begin{corollary}[{\cite{gaboriau-lyons}}]\label{lamplighter-cor}
Let $\Gamma$ be a non-amenable finitely generated group. Then there is a $k\in \mathbb{Z}$ such that the orbit equivalence relation $\mathcal{R}_{\Gamma}$ of the Bernoulli action of $\Gamma$ on $\left( \mathbb{Z}/k\mathbb{Z}, \nu_k \right)^{\Gamma}$ contains the orbit equivalence relation of an essentially free, measure-preserving action of $\mathbb{F}_2$.
\end{corollary}

\subsection{Non-dimension flatness}
In this section we show how one can obtain non-dimension flat inclusions from the Gaboriau-Lyons theorem discussed above. The main result is as follows.
\begin{thm}\label{non-flat-thm-all-borel-sets}
Let $\Gamma$ be  a finitely generated non-amenable group  and let $A_0$ be the finite cyclic group obtained from Corollary \ref{lamplighter-cor}. Denote by $A$ the direct sum $\oplus_\Gamma A_0$, by $\hat{A}$ its Pontryagin dual and by $\B_\all$ the system of all Borel subsets in $\hat{A}$. Then
the inclusion $\CC[\B_\all]\rtimes \Gamma \subseteq L^\infty(\hat{A})\rtimescom \Gamma$ is not dimension  flat.
\end{thm}

\begin{proof} 
Denote by $\G$ the translation groupoid of the Bernoulli action of $\Gamma$ on $\hat{A}$ and by $\H$ the sub-groupoid arising from the action of $\FF_2$. We now get

\begin{align*}
1&=\beta_1^{(2)}(\FF_2)  \\
&=\dim_{L\H}\Tor_1^{\CC \H}(L\H, L^\infty(\hat{A})) \tag{\cite[Theorem 5.5]{sauer-betti-of-groupoids}}\\
&=\dim_{L\G}\Tor_1^{\CC \H} (L\G, L^\infty(\hat{A})) \tag{\cite[Theorem 6.29]{Luck02}}\\
&=\dim_{L\G}\Tor_1^{\CC \G} \left(L\G, \CC\G\underset{{\CC \H}} {\tens} L^\infty(\hat{A})\right) \tag{Proposition \ref{prop:dimflatness} \& Corollary \ref{lma:flatbasechange}}\\
&=\dim_{L\G}\Tor_1^{L^\infty(\hat{A})\rtimes \Gamma } \left(L\G, \CC\G\underset{{\CC \H}} {\tens} L^\infty(\hat{A})\right). \tag{\cite[Theorem 4.11]{sauer-betti-of-groupoids}}
\end{align*}
In the last line of the above computation we may,  by \cite[Lemma 4.1]{sauer-betti-of-groupoids}, replace  the module $\CC\G\tens_{\CC \H} L^\infty(\hat{A})$  with any other $L^\infty(\hat{A})\rtimes \Gamma$-module which is $\dim_{L^\infty(\hat{A})}$-isomorphic to it without changing the $L\G$-dimension of the $\Tor$-module. In order to apply \cite[Lemma 4.1]{sauer-betti-of-groupoids} we need to know that $L^\infty(\hat{A})\rtimes \Gamma$ is dimension-compatible as an $L^\infty(\hat{A})$-bimodule, but this follows from the remarks  proceeding Lemma \ref{completion-properties}. Appealing to Proposition \ref{the-first-iso}, we  have a dimension-isomorphism
\begin{align*}
\CC\G\underset{{\CC \H}} {\tens} L^\infty(\hat{A}) \simeq \CC[\G/\H].
\end{align*}
Furthermore, the decomposition lemma shows that $\CC[\G]$ is rank dense in $\CC[\G]_t$ and since $\kappa_\H^\G$ is an $L^\infty(\hat{A})$-homomorphism (in particular rank continuous) the inclusion $\CC[\G/\H]\subseteq \CC[\G/\H]_t$ is also a $\dim_{L^\infty(\hat{A})}$-isomorphism. Thus we obtain that
\begin{align}\label{den-nye-et-formel}
1= \dim_{L\G}\Tor_1^{L^\infty(\hat{A})\rtimes \Gamma } \left(L\G, \CC[\G/\H]_t\right).
\end{align}
Since $L^\infty(\hat{A})\rtimescom \Gamma =L\G$, if the inclusion $\mathbb{C}[\B_\all]\rtimes \Gamma \subseteq L^\infty(\hat{A})\rtimescom \Gamma$ were dimension flat then, for an arbitrary $\mathbb{C}[\B_\all]\rtimes \Gamma$-module $K$, we would have 

\begin{align*}
0&=\dim_{L^\infty(\hat{A})\rtimescom \Gamma}\Tor_1^{\mathbb{C}[\B_\all]\rtimes \Gamma}(L^\infty(\hat{A})\rtimescom \Gamma, K)\\
&=\dim_{L\G} \Tor_1^{\mathbb{C}[\B_\all]\rtimes \Gamma}(L\G, K) \\
&=\dim_{L\G} \Tor_1^{L^\infty(\hat{A})\rtimes \Gamma }\left(L\G, L^\infty(\hat{A})\rtimes\Gamma\underset{\mathbb{C}[\B_\all]\rtimes \Gamma}{\tens} K\right) ,
\end{align*}
where the last equation follows from the dimension flat base change formula in Corollary \ref{dim-flat-base-change-v2}.  In order to prove that  $\mathbb{C}[\B_\all]\rtimes \Gamma \subseteq L^\infty(\hat{A})\rtimescom \Gamma$ is not dimension  flat it therefore suffices, by \eqref{den-nye-et-formel},  to show that there exists a $\mathbb{C}[\B_\all]\rtimes \Gamma$-module $K$ and a homomorphism of left $L^\infty(\hat{A})\rtimes \Gamma$-modules
\begin{align}\label{mult-iso}
L^\infty(\hat{A})\rtimes\Gamma\underset{{\mathbb{C}[\B_\all]\rtimes \Gamma}}{\tens} K  \To   \CC[\G/\H]_t, 
\end{align}
which is a $\dim_{L^\infty(\hat{A})}$-isomorphism.  
To this end, we define 
\[
K:=\spann_\CC\{ \bbb_E \mid E\subseteq \G/\H \text{ and } \bbb_E\in \CC[\G/\H]_t   \} \subseteq \CC[\G/\H]_t
\] 
It is easy to see that this becomes a module for the action (via $\kappa_\H^\G$) of $\CC[\B_\all]\rtimes \Gamma\subseteq \CC[\G]$ 
and we now claim
that the multiplication map
\[
\mult\colon L^\infty(\hat{A})\rtimes\Gamma\underset{{\mathbb{C}[\B_\all]\rtimes \Gamma}}{\tens} K  \To   \CC[\G/\H]_t 
\]
is a $\dim_{L^\infty(\hat{A})}$-isomorphism.  \\

To see that $\mult$ is dimension-surjective, observe first that
\[
\kappa_\H^{\G}(L^\infty(\hat{A})\rtimes \Gamma)= (L^\infty(\hat{A})\rtimes \Gamma).\bbb_{\hat{A}} \subseteq \rg(\mult).
\]
By  \cite[Lemma 5.4]{sauer-betti-of-groupoids}  and Lemma \ref{lma:decomposition} the inclusions  
\[
L^\infty(\hat{A})\rtimes \Gamma \subseteq \CC[\G]\subseteq \CC[\G]_t
\]
are $\dim_{L^\infty(\hat{A})}$-isomorphism and $\kappa_{\H}^\G(L^\infty(\hat{A})\rtimes \Gamma)$ is therefore rank dense in $\kappa_\H^\G(\CC[\G]_t)=\CC[\G/\H]_t$. Thus $\mult$ is dimension-surjective.\\

To prove dimension-injectivity, let $T\in \ker(\mult)$ be given. We actually aim to prove that $\mult$ is properly injective; i.e.~that $T=0$. Write $T$ as $\sum_{i=1}^n f_i \tens m_i$ where $f_i\in L^\infty(\hat{A})$ and $m_i\in K$;
since the family of target-bounded Borel subsets in $\G/\H$ is stable 
under finite intersections, we can find mutually disjoint, target-bounded, Borel subsets   $F_1,\dots, F_r$ in $\G/\H$ such that each $m_i$ can be written as
\[
m_i=\sum_{j=1}^r m_i(F_j)\bbb_{F_j}
\]
for some $m_i(F_j)\in \CC$.  Since $T\in \ker(\mult)$ and $\kappa_{\H}^{\G}$ acts like the identity on $\CC[\G/\H]_t$ we have
\[
0= \mult(T)=\kappa_{\H}^{\G}\left(\sum_{i=1}^n f_i\ast m_i\right)=\sum_{j=1}^r \left(\sum_{i=1}^n f_i m_i(F_j)\right)\ast \bbb_{F_j}. 
\]
As the $F_j$'s are disjoint this implies that the restriction of $\sum_{i=1}^n f_i m_i(F_j)$ to $t(F_j)$ is zero for every $j\in\{1,\dots, r\}$. Moreover, since each of the $\bbb_{F_j} \in K$ we may rewrite $T$ as
\begin{align*}
T &= \sum_{j=1}^r \left(\sum_{i=1}^n f_i m_i(F_j)\right)\tens \bbb_{F_j}\\
&= \sum_{j=1}^r \left(\sum_{i=1}^n f_i m_i(F_j)\right)\tens \bbb_{t(F_j)}\ast \bbb_{F_j}\\
&=\sum_{j=1}^r\left(\sum_{i=1}^n f_i m_i(F_j)\right)\bbb_{t(F_j)}\tens \bbb_{F_j}=0. \qedhere
\end{align*} 
\end{proof}

We remark that a converse to Theorem \ref{amenability-to-dim-flat-thm} could have been obtained  without reference to the finite group $A_0$, by simply using the continuous base space version (Theorem \ref{gaboriau-lyons-continuous}) of Gaboriau-Lyons' theorem in the statement and  proof of Theorem \ref{non-flat-thm-all-borel-sets}. However, in the following section we will investigate how ``close'' to a group algebra we can choose the crossed product $\CC[\B]\rtimes \Gamma$ exhibiting the non-dimension flatness, and the construction above shows that at least the measure space can be chosen, naturally, to arise from a discrete group. Returning to  the purely measure theoretic context we obtain the following groupoid solution to L{\"u}ck's amenability conjecture. 
\begin{por}\label{groupoid-porism}
A discrete group $\Gamma$ is amenable if and only if the following holds: for any free, ergodic, p.m.p.~action of $\Gamma$ on a non-atomic standard Borel space $(X,\mu)$ the inclusion of the corresponding groupoid ring $\CC[\mathcal{R}_{\Gamma \curvearrowright X}]$ into the  groupoid von Neumann algebra $L(\mathcal{R}_{\Gamma \curvearrowright X})$ is dimension  flat.
\end{por}
\begin{proof}
If $\Gamma$ is amenable then by \cite[Corollary 6.6]{amenability-and-vanishing} the inclusion $L^\infty(X)\rtimes \Gamma \subseteq L^\infty(X)\rtimescom \Gamma$ is  dimension  flat and applying \cite[Theorem 4.11]{sauer-betti-of-groupoids} we get, for an arbitrary $\CC[\mathcal{R}_{\Gamma \curvearrowright X}]$-module $K$ and $p\geqslant 1$, that 
\begin{align*}
\dim_{L(\mathcal{R}_{\Gamma \curvearrowright X})}\Tor_p^{\CC[\mathcal{R}_{\Gamma \curvearrowright X}]}(L(\mathcal{R}_{\Gamma \curvearrowright X}), K) &= \dim_{L(\mathcal{R}_{\Gamma \curvearrowright X})}\Tor_p^{L^\infty(X)\rtimes \Gamma}(L(\mathcal{R}_{\Gamma \curvearrowright X}), K)\\
&=\dim_{L^\infty(X)\rtimescom \Gamma}\Tor_p^{L^\infty(X)\rtimes \Gamma}(L^\infty(X)\rtimescom \Gamma, K)=0.
\end{align*}
Conversely, if $\Gamma$ is not amenable then by Theorem \ref{gaboriau-lyons-continuous} the Bernoulli action of $\Gamma$ on $X:=[0,1]^\Gamma$ contains  a free action of $\FF_2$; hence we have $\CC[\mathcal{R}_{\FF_2 \curvearrowright X}  ] \subseteq \CC[\mathcal{R}_{\Gamma \curvearrowright X}]$ and, like in the proof of Theorem \ref{non-flat-thm-all-borel-sets}, we therefore get

\begin{align*}
1&=\beta_1^{(2)}(\FF_2)  \\
&=\dim_{L(\mathcal{R}_{\FF_2 \curvearrowright X})}\Tor_1^{\CC[\mathcal{R}_{\FF_2 \curvearrowright X}]}(L(\mathcal{R}_{\FF_2 \curvearrowright X}), L^\infty(X)) \tag{\cite[Theorem 5.5]{sauer-betti-of-groupoids}}\\
&=\dim_{L(\mathcal{R}_{\Gamma \curvearrowright X})}\Tor_1^{\CC[\mathcal{R}_{\FF_2 \curvearrowright X}]} (L(\mathcal{R}_{\Gamma \curvearrowright X}), L^\infty(X)) \tag{\cite[Theorem 6.29]{Luck02}}\\
&=\dim_{L(\mathcal{R}_{\Gamma \curvearrowright X})}\Tor_1^{\CC[\mathcal{R}_{\Gamma \curvearrowright X}]} \left( L(\mathcal{R}_{\Gamma \curvearrowright X}), \CC[\mathcal{R}_{\Gamma \curvearrowright X}]\underset{{\CC[\mathcal{R}_{\FF_2 \curvearrowright X}]}} {\tens} L^\infty(X)\right), 
\end{align*} 
where the last equality follows from Proposition \ref{prop:dimflatness} and Corollary \ref{lma:flatbasechange}.
Thus the inclusion $\CC[\mathcal{R}_{\Gamma \curvearrowright X}]\subseteq L(\mathcal{R}_{\Gamma \curvearrowright X})$ cannot be dimension  flat.
\end{proof}

\subsection{Improving the subalgebra}
In the previous section we saw that whenever $\Gamma$ is  a finitely generated non-amenable group then there exists a finite abelian group $A_0$ such that the Bernoulli action of $\Gamma$ on the dual $\hat{A}$ of $A:=\oplus_\Gamma A_0$ contains an action of $\FF_2$, and as a consequence the inclusion $\CC[\B_\all]\rtimes \Gamma \subseteq L^\infty(\hat{A})\rtimescom \Gamma$ is not dimension  flat. 
It would of course be desirable to be able to replace $\CC[\B_\all]$ with $\CC[\B_\co]$ and thereby obtain non-dimension  flatness of the, somewhat more natural, inclusion $\CC[A_0\wr \Gamma]\subseteq L(A_0\wr \Gamma)$.
Although we were not able to show this, certain improvements are still possible.  As a first step we show that one can replace  $\CC[\B_\all]$ with an algebra of step functions with a countable linear basis.

\begin{prop}\label{fg-ground-ring}
Let $\Gamma$ be finitely generated and non-amenable and let $A_0$, $A$ and $\hat{A}$ be as above. Then there exists a countable $\B\subseteq \B_\all$ such that $\CC[\B]\rtimes \Gamma$ is finitely generated as a $\CC\Gamma$-module and  for which the inclusion $\CC[\B]\rtimes \Gamma \subseteq L^\infty(\hat{A})\rtimescom \Gamma$ is not dimension  flat.
\end{prop}
\begin{proof}
By Theorem \ref{non-flat-thm-all-borel-sets} there exists a $\CC[\B_\all]\rtimes\Gamma$-module $L$ such that
\[
\dim_{L^\infty(\hat{A})\rtimescom \Gamma}\Tor_1^{\CC[\B_\all] \rtimes \Gamma}(L^\infty(\hat{A})\rtimescom \Gamma, L)>0,
\]
and by an inductive limit argument\footnote{See e.g.~the last part of the proof of Theorem 6.37 in \cite{Luck02} for the details.}  we may assume that $L$ is finitely presented. We can therefore find a presentation
\[
(\CC[\B_\all] \rtimes \Gamma)^k \overset{\cdot T }{\To} (\CC[\B_\all] \rtimes \Gamma)^l \To L \To 0,
\]
where $T=(T_{ij})$ is a $k\times l$ matrix with entries from $\CC[\B_\all]\rtimes \Gamma$.  Hence there exists a finite Borel partition $F_1,\dots, F_r$ of $\hat{A}$ such that every element $T_{ij}$ can be written as
\[
T_{ij}=\sum_{k=1}^r \bbb_{F_k}\left(\sum_{\gamma\in S_k} r_\gamma^{(ij)}u_\gamma\right) 
\]
for some finite subsets $S_k\subseteq \Gamma$ and some $r_\gamma^{(ij)}\in \CC$.  If there are at least two $F_i$'s we
 define $\B$ to be the family of subsets obtained by closing the finite family
\[
\{ F_1,\dots, F_r \}
\]
under finite intersections, complements and $\Gamma$-translates.  If there is only one $F_i$ we simply add an artificial subset $F_0$ with measure neither zero nor one and close $\{F_0,F_1\}$ under complements, finite intersections and $\Gamma$-translates. Since the Bernoulli action of $\Gamma$ is free, ergodic and p.m.p.~the crossed product von Neumann algebra $L^\infty(\hat{A})\rtimescom\Gamma$ is a $\twoone$-factor and hence the assumptions in Corollary \ref{dim-flat-base-change-v2} are satisfied.  Since $T_{ij}\in \CC[\B]\rtimes\Gamma$ we have, by right-exactness of the tensor product, that
\begin{align*}
L^\infty(\hat{A})\rtimes \Gamma \underset{\CC[\B_\all]\rtimes \Gamma}{\tens} L &= L^\infty(\hat{A})\rtimes \Gamma \underset{\CC[\B_\all]\rtimes \Gamma}{\tens} \frac{(\CC[\B_\all] \rtimes \Gamma)^l}{(\CC[\B_\all]\rtimes \Gamma)^kT}\\
&=\frac{(L^\infty(\hat{A})\rtimes \Gamma)^l }{(L^\infty(\hat{A})\rtimes \Gamma)^k T}\\
&= L^\infty(\hat{A})\rtimes \Gamma \underset{\CC[\B]\rtimes \Gamma}{\tens} \underbrace{\frac{(\CC[\B] \rtimes \Gamma)^l}{(\CC[\B]\rtimes \Gamma)^kT}}_{=:L'}
\end{align*}
Using the dimension  flat base change formula (Corollary \ref{dim-flat-base-change-v2}) twice we therefore obtain

\begin{align*}
0 &<\dim_{L^\infty(\hat{A})\rtimescom \Gamma}\Tor_1^{\CC[\B_\all] \rtimes \Gamma}  \left(L^\infty(\hat{A})\rtimescom \Gamma, L\right)\\
&=  \dim_{L^\infty(\hat{A})\rtimescom \Gamma}\Tor_1^{L^\infty(\hat{A}) \rtimes \Gamma}\left(L^\infty(\hat{A})\rtimescom \Gamma, L^\infty(\hat{A}) \rtimes \Gamma\tens_{\CC[\B_\all]\rtimes \Gamma} L\right)\\
&=  \dim_{L^\infty(\hat{A})\rtimescom \Gamma}\Tor_1^{L^\infty(\hat{A}) \rtimes \Gamma}\left(L^\infty(\hat{A})\rtimescom \Gamma, L^\infty(\hat{A}) \rtimes \Gamma\tens_{\CC[\B]\rtimes \Gamma} L'\right)\\
&=\dim_{L^\infty(\hat{A})\rtimescom \Gamma}\Tor_1^{\CC[\B] \rtimes \Gamma}\left(L^\infty(\hat{A})\rtimescom \Gamma, L'\right)
\end{align*}
\end{proof}

It would be desirable to have more information about the algebra $\mathbb{C}[\mathscr{B}]\rtimes \Gamma$ from the previous proposition. Ideally, we would like to know whether or not we can replace it with the group algebra of $A_0\wr \Gamma$, or if it is the complex group algebra of any countable discrete group. In fact, we do not know of any general criteria to decide whether an algebra of this form is a group algebra or not. The following proposition provides such a criterion.

\begin{proposition}
Let $\Gamma \curvearrowright (X,\mu)$ be a measure preserving action on a standard probability space such that each element $\gamma \in \Gamma\setminus \{ \bbb\}$ acts ergodically.
Let $X=\sqcup_{n=1}^N F_n$ be a finite partition of $X$ and let $R$ be the $\Gamma$-invariant unital $*$-algebra generated by the step functions $\bbb_{F_n}$. Then there exists a finite abelian group $A_0$ such that $\CC[A_0\wr \Gamma] \simeq R\rtimes \Gamma \subseteq L^{\infty}X\rtimescom \Gamma$ if for every $m\neq n$ and every $\gamma \neq \bbb$ we have $\mu(F_m\cap \gamma(F_n)) > 0$.
\end{proposition}

\begin{proof}
Let $\zeta$ be a primitive $N$'th root of unity and let $A_0=\langle \zeta\rangle = \mathbb{Z}/N\mathbb{Z}$. We claim that the assignment $\zeta \mapsto u:=\sum_{n=1}^{N} \zeta^n \bbb_{F_n}$ extends to the desired isomorphism.

The map is seen to be surjective. To see injectivity we must show that, enumerating $\Gamma=\{\gamma_k\}_{k\in \mathbb{N}}$, for any $K\in \mathbb{N}$
\begin{equation}
\dim_{\mathbb{C}} \operatorname{span}_{\mathbb{C}} \{ \gamma_1(u^{i_1})\cdots \gamma_K(u^{i_K}) \mid 1\leq i_k\leq N \} = N^K. \nonumber
\end{equation}
But this is clear because the subspace is spanned linearly by 
\[
\{ \bbb_{\cap_{i=1}^{K}\gamma_i(F_{j(i)})}\mid j\colon \{1,\dots, K \} \to \{1,\dots, N\} \}
\]
and the hypothesis is seen to imply that these are linearly independent.
\end{proof}

\begin{corollary}
Let $\Gamma \curvearrowright (X,\mu)$ be a measure preserving action on a standard probability space by homeomorphisms and such that each element $\gamma \in \Gamma\setminus \{ \bbb\}$ acts ergodically.
Let $R$ be the $\Gamma$-invariant unital $*$-algebra generated by the two step functions $\bbb_C$ and $\bbb_U$ where $C\subseteq X$ is a compact set with empty interior, $\mu(C)\geq 1/2$, and $U=C^{\complement}$ is an open dense set. Then $R\rtimes \Gamma \simeq \mathbb\CC[{\ZZ}/2\mathbb{Z} \wr \Gamma]$.
\end{corollary}

\def\cprime{$'$} \def\cprime{$'$}

\end{document}